\documentclass[10 pt]{amsart}
\usepackage{amsfonts,amsthm,amsmath,amssymb,fullpage,graphicx,bm,xcolor,bbm,mathrsfs,algorithm2e}
\usepackage[utf8]{inputenc}

\usepackage[british]{babel}
\usepackage[shortlabels]{enumitem}
\usepackage[font=sf]{caption}

\allowdisplaybreaks

\usepackage{hyperref} 
\hypersetup{colorlinks=true, linkcolor=blue, citecolor=blue, filecolor=blue, urlcolor=blue}
\usepackage[numbers, comma, sort&compress]{natbib}

\usepackage[shortlabels]{enumitem}
\usepackage[comma,sort&compress]{natbib}

\definecolor{greennew}{rgb}{0.1
	,0.4,0.1}
\newcommand{\W}{\mathcal{W}}
\newcommand{\E}{\mathbb{E}}
\newcommand{\N}{\mathbb{N}}

\renewcommand{\P}{\mathbb{P}}

\newcommand{\A}{\mathcal{A}}

\newcommand{\R}{\mathbb{R}}
\renewcommand{\l}{\langle}
\renewcommand{\r}{\rangle}

\title{\bf{Random Walk in Slowly Changing Environments}}

\author{Bryan Park}\thanks{Stanford University.  
	bryan314@stanford.edu}
\author{Souvik Ray}\thanks{School of Data Science and Society, University of North Carolina at Chapel Hill. souvikr@unc.edu}

\numberwithin{equation}{section}

\date{}
\usepackage{cleveref}
\newtheorem{theorem}{Theorem}[section]
\newtheorem{definition}[theorem]{Definition}
\newtheorem{result}[theorem]{Result}

\newtheorem{corollary}[theorem]{Corollary}
\newtheorem{remark}[theorem]{Remark}
\newtheorem{lemma}[theorem]{Lemma}

\crefname{equation}{}{}
\crefname{theorem}{Theorem}{Theorems}
\crefname{assumption}{Assumption}{Assumptions}
\crefname{remark}{Remark}{Remarks}
\Crefname{lemma}{Lemma}{Lemmas}
\crefname{lemma}{Lemma}{Lemmas}
\crefname{enumi}{}{}

\begin{document}
	\begin{abstract}
		A \textit{Random Walk in Changing Environment} (RWCE) is a weighted random walk on a locally finite, connected graph $G$ with random, time-dependent edge-weights. This includes self-interacting random walks, where the edge-weights depend on the history of the process. In general, even the basic question of recurrence or transience for RWCEs is difficult, especially when the underlying graph contains cycles. In this note, we derive a condition for recurrence or transience that is too restrictive for classical RWCEs but instead works for any graph $G.$ Namely, we show that any bounded RWCE on $G$ with ``slowly" changing edge-weights inherits the recurrence or transience of the initial weighted graph.
	\end{abstract}
	
		\maketitle
		
	\section{Introduction}
	The main focus of inspection in this paper is \textit{Random Walk in Changing Environment} or RWCE as introduced by \cite{Amir2020}. Broadly speaking, an RWCE is a random walk on a graph where each step is performed in a different \textit{environment}, i.e., the weights corresponding to the edges of the graph change over time. The rigorous definition of this setup can be formulated as follows. We shall work on a simple, undirected, locally finite\footnote{A graph is called locally finite if all of its vertices have finite degree.} and connected graph $G=(V,E)$ with vertex set $V$ and edge set $E$. In particular, these conditions guarantee that both $V$ and $E$ are countable. For any $x,y \in V$, the notation $x \sim y$ means that $x$ and $y$ are neighbors in $G$.
	
	\begin{definition}[RWCE]{\label{def1}}
		Fix an underlying probability space $(\Omega, \mathscr{F}, \mathbb{P})$. A \textit{Random Walk in Changing Environment }(RWCE) on the graph $G=(V,E)$ is a stochastic process $\{\l X_t, G_t\r\}_{t=0}^{\infty}$, where $G_t=(V,E,C_t)$ are graphs over the fixed vertex set $V$ and edge set $E$ with corresponding weight configurations given by measurable $C_t : \Omega \to \mathcal{W}:=[0,\infty)^E$, satisfying the following dynamics: For all $t\in \mathbb{N}_0 := \left\{0,1,2,\ldots \right\}$ and $y \in V$, we have
		\begin{equation}{\label{dynamics1}}
			\P[X_{t+1}=y\mid \mathscr{F}_t]=\frac{C_t(X_t,y)}{\sum_{z\,:\,z\sim X_t}C_t(X_t,z)}\mathbbm{1}_{\left(y \sim X_t \right)},
		\end{equation}
		where $X_t : \Omega \to V$ is measurable and $\mathscr{F}_t :=\sigma \left( \l X_s, G_s \r : s=0,1,\ldots,t \right) $.
	\end{definition}
	
Henceforth, $C_t(x,y)=C_t(y,x)$ will denote the weight at time $t$ associated with the edge connecting the vertices $x$ and $y$. For notational simplicity, we shall write $P(x,y;C)$ to denote the transition probability from $x$ to $y$ when the underlying weight configuration is $C$, i.e.,
$$ P(x,y;C) := \dfrac{C(x,y)}{\sum_{z: z \sim x} C(x,z)} \mathbbm{1}_{(x \sim y)}.$$
	
The sequence $\left\{X_t\right\}_{t=0}^{\infty}$ is called the \textit{Random Walk}, while the sequence $\left\{G_t\right\}_{t=0}^{\infty}$ is called the \textit{Environment}.	In words, at time $t\in\N_0$, the RWCE traverses a neighboring edge from $X_t$ with probability proportional to its weight at that time, given by the realization of $C_t$, which may depend on the history of the process so far and on extra randomness. In the rest of the paper, $\mathscr{G}_t:= \sigma(G_s : 0 \leq s \leq t)$ will denote the $\sigma$-algebra containing all the information of the environment until time $t$, whereas $\mathscr{G}_{\infty} := \sigma(G_s : s \geq 0)$.

While the term RWCE was coined by \cite{Amir2020}, many special cases have been extensively studied before. For instance, RWCEs include the large class of self-interacting random walks, where the weights depend on the history of the process. A well-known class of examples regard the \textit{reinforced random walk} where traversing an edge changes the weight of that edge. Reinforced random walks were first introduced by \cite{Diaconis1986} and later generalized by \cite{Davis1990} and \cite{Pemantle2007}. The broad class of reinforced random walks include \textit{once-reinforced random walks} (ORRW) where the weight of a particular edge is altered the first time it is traversed; \textit{linearly edge-reinforced random walks} (LRRW) where the weight of any edge is increased every time it is traversed. We refer to the survey articles by \cite{Gady2012} and \cite{Pemantle2007} for more details on these models and their properties. Other examples include the \textit{Bridge burning random walk} \cite{Amir2020}, \textit{Laplacian random walk}  \cite{Lawler1979} and \textit{self-avoiding walk with bond repulsion}  \cite{Toth1995}.

	
We are interested in the question of recurrence or transience for RWCEs which is, in general, difficult to answer  as the process need not be Markovian. In fact, in contrast to random walks in a fixed environment, different notions of recurrence and transience may not even coincide for RWCEs in general. For example, it is possible to construct an RWCE which almost surely visits some vertex infinitely often, while visiting some other vertex finitely often, see~\cite[Example 1]{Park2024} for one such example. In this note, we follow and adopt the strongest definitions of recurrence and transience from \cite{Amir2020}.
	\begin{definition}[Recurrence/Transience/Mixed-Type]
		An RWCE is \textit{recurrent} if almost surely every vertex is visited infinitely often. It is \textit{transient} if almost surely every vertex is visited finitely often. Otherwise, it is of \textit{mixed-type}.
	\end{definition}
There have been quite a few works studying recurrence and transience for reinforced random walks. It has been conjectured that the ORRW is recurrent on $\mathbb{Z}^d$ for $d=1,2$ and undergoes a phase transition for $d \geq 3$, being recurrent when the reinforcement parameter  is large and transient when it is small. These kind of questions on $\mathbb{Z}^d$ for $d \geq 2$ remain completely open. The first example of such a phase transition was proved in \cite{Kious2018} for a particular class of trees with polynomial growth, which is in contrast with the result of \cite{Durrett2002} that the ORRW is transient on regular trees. Later, \cite{Collevecchio2018} provided a complete picture of recurrence/transience and phase transition for ORRWs on trees, characterizing the critical parameter by a quantity called the \textit{branching-ruin number}. For results on ORRWs when the underlying graph is $\mathbb{Z}$ or a ladder of the form $\mathbb{Z} \times \Gamma$, where $\Gamma$ is some finite graph, we refer to the works of \cite{Kious2018s} and \cite{Pfaffelhuber2021}. In contrast to these, recurrence and transience for LRRWs were established in a series of works by \cite{Angel2014}, \cite{Disertori2015} and \cite{Sabot2015}.

Once we shift our attention from special type of RWCEs like reinforced random walks, we find very few results on recurrence/transience for general RWCEs.  Without any further assumptions on the RWCE, there are myriad possibilities for recurrence and transience, as demonstrated by a series of examples in \cite{Amir2020}. To arrive at some concrete results, \cite{Amir2020} considered a few restrictions on the RWCE as mentioned below.

	\begin{definition}{\label{proper}}
		An RWCE is said to be proper if $C_t(e) \in (0,\infty)$ for all $e \in E$ and $t \geq 0$. It is called improper otherwise.
	\end{definition}
	
	\begin{definition}{\label{bounded}}
		An RWCE on the graph $G=(V,E)$ is said to be bounded from above (resp.~below) 
		if $\sup_{t \geq 0}C_t(e) < \infty$ (resp.~$\inf_{t \geq 0} C_t(e) >0$) for all $e \in E$ and for all $t \geq 0$, almost surely.
	\end{definition}
	
	\begin{remark}
		Definition~\ref{bounded} is a weaker form of boundedness, compared to requiring $C_t(e)$ to be bounded from above (resp.~below) by some finite (resp.~non-zero) constant uniformly in time, almost surely.
	\end{remark}

\begin{definition}{\label{monotone}}
	An RWCE is called monotone increasing (resp.~decreasing) if for all $t \geq 0$, we have $C_{t+1} \geq C_t$ (resp.~$C_{t+1} \leq C_t$) almost surely.
	\end{definition}
	
\begin{definition}{\label{elliptic}}
	An RWCE $\left\{\l X_t,G_t \r \right\}_{t \geq 0}$ on the graph $(V,E)$ is said to be elliptic (uniformly in time) if for any $\left\{x,y\right\} \in E$, there exists a $\mathscr{G}_{\infty}$-measurable random variable $P_{x,y}$ with $\P[P_{x,y}>0]=1$ such that for all $t \geq 0$, we have $\mathbb{P}\left(X_{t+1}=y \mid \mathscr{F}_t\right) \geq P_{x,y}$ on the event $(X_t=x)$ almost surely. In other words, $ P(x,y;C_t) \geq P_{x,y}\mathbbm{1}_{(X_t=x)}$, for all  $t \geq 0$, almost surely.
\end{definition}
			Throughout this paper, we shall restrict our attention to proper RWCEs. It is easy to see that a bounded (from above and below) RWCE is also elliptic. The usefulness of the concept of ellipticity lies in the fact that an elliptic RWCE which visits some vertex infinitely (resp.~finitely) often almost surely is recurrent (resp.~transient). The argument goes as follows. Since $G$ is connected, it is enough to show that if the RWCE visits some vertex infinitely (resp.~finitely) often almost surely, then it also visits its neighbors infinitely (resp.~finitely) often almost surely. Take $x \sim y \in V$. The ellipticity condition guarantees that for any $T \geq 0$, $\P[X_t=y, \text{ for some } t \geq T \mid \mathscr{F}_T] \geq P_{x,y}$, on the event $(X_T=x)$ almost surely. Taking $T \to \infty$ and applying \textit{Levy's Upward Theorem}, we see that $\mathbbm{1}(X_t=y \text{ infinitely often}) \geq P_{x,y} \mathbbm{1}(X_t=x \text{ infinitely often})$, almost surely. Thus, on the event that $x$ is visited by the walk infinitely often, $y$ is also visited infinitely often almost surely. Hence, if $x$ is visited infinitely often almost surely, the same is true for $y$ and if $y$ is visited finitely often almost surely, the same is true for $x$, which proves our claim.  In the rest of this paper, we shall focus on proper, elliptic RWCEs which are bounded from above. We note that many of the interesting examples in the literature fall in this category.
			
			Another crucial categorization of RWCEs is obtained by considering whether the dynamics underlying the evolution of the environment depends on the history of the walk or not. The following definition makes the concept rigorous. 
			
			\begin{definition}
			An RWCE $\left\{\l X_t, G_t \r \right\}_{t \geq 0}$ is said to be non-adaptive if the distribution of $C_{t+1}$ given $(X_0,C_0,\ldots,X_t,C_t,X_{t+1})$ does not depend of $(X_0,\ldots,X_{t+1})$. More precisely,  the walk is called non-adaptive if for any measurable $\mathcal{A} \subseteq \mathcal{W} = [0,\infty)^E$, we have 
			$$ \mathbb{P} \left[ C_{t+1} \in \mathcal{A} \mid \sigma \left( \mathscr{F}_t, X_{t+1} \right) \right] = \mathbb{P} \left[ C_{t+1} \in \A \mid \mathscr{G}_t \right], \; \text{almost surely}, \; \text{for all } t \geq 0.$$
			The walk is called adaptive otherwise.
		\end{definition}
		
		Since the dynamics of the environment does not depend on the path of the walk for a non-adaptive RWCE, one can construct such a walk in a hierarchical fashion such that conditional on the environment (at all time points), the walk is just an RWCE with deterministic weight configurations. Lemma~\ref{nonadaptive} makes this observation concrete. The proof is deferred to Section~\ref{appn}. 
		
		\begin{lemma}{\label{nonadaptive}}
			Suppose $\left\{\l X_t, G_t \r\right\}_{t \geq 0}$ is a non-adaptive RWCE on the graph $G=(V,E)$. Consider a new process $\left\{X^{\prime}_t : t \geq 0\right\}$ on $V$, defined on the same probability space, such that $X_0^{\prime}=X_0$ and for any $t \geq 0$, $y \in V$, we have
			$$ \mathbb{P} \left( X^{\prime}_{t+1}=y \mid \sigma\left(\mathscr{G}_{\infty}, \mathscr{F}_t^{\prime} \right) \right) = \dfrac{C_{t}(X^{\prime}_t,y)}{\sum_{z \sim X_t^{\prime} }C_t(X^{\prime}_t,y)} \mathbbm{1} \left(y \sim X^{\prime}_t \right) = P(X_t^{\prime},y; C_t),$$
				where we define $\mathscr{F}^{\prime}_t := \sigma \left( \l X^{\prime}_s,G_s \r : 0 \leq s \leq t \right).$ Then, the new RWCE $\left\{\l X^{\prime}_t, G_t \r\right\}_{t \geq 0}$ has the same law as the original RWCE $\left\{\l X_t, G_t \r\right\}_{t \geq 0}$.
		\end{lemma}
		
	In their work, \cite{Amir2020} showed that neither boundedness nor monotonicity are enough to draw any significant conclusions about the recurrence/transience of RWCEs, at least in the adaptive setting, even for a simple situation where the underlying graph is $\mathbb{Z}$. Examples were also provided to demonstrate that even in the non-adaptive setting, boundedness (even by deterministic weight configurations) does not imply recurrence/transience. Nevertheless, they were able to establish that any monotone increasing adaptive RWCE on the graph $\mathbb{Z}$, bounded above by some recurrent\footnote{A weighted graph $(G,w)$ is recurrent (resp.~transient) if the weighted random walk on $(G,w)$ is recurrent (resp.~transient).}  connected\footnote{A weighted graph $(G,w)$ is connected if the weighted random walk can reach any destination vertex from any starting vertex in finite time with positive probability.} graph, is recurrent. This particular result appears in \cite[Theorem 4.1]{Amir2020} and was also extended to cover the cases where the underlying graph is a tree, see \cite[Theorem 5.1]{Amir2020}. Similar results were shown regarding transience, again when the underlying graph is a tree. We refer to \cite[Theorem 4.1-4.4, 5.1, 5.2]{Amir2020} for precise statements of these results. In contrast, these results can not be extended to more complex graphs within the adaptive setting, as demonstrated by an example in \cite[Example 6.1]{Amir2020} for $\mathbb{Z}^2$. However, the authors conjectured that some of these aforementioned results, namely \cite[Theorem 4.1-4.4]{Amir2020}, hold true on any graph in the non-adaptive setting. These results remain as conjectures as per our current knowledge, although it is worth mentioning the recent works of \cite{Dembo2014f, Dembo2014s}, where the authors have analyzed random walks on ``monotone domains" with a focus on $\mathbb{Z}^d$ and proved criteria for recurrence/transience for such walks in both the non-adaptive \cite{Dembo2014f} and adaptive \cite{Dembo2014s} setting.
	
    In this note, we derive a condition for recurrence or transience that is too restrictive for classical RWCEs (such as the ORRW) but instead holds for any graph. Loosely speaking, we show that any proper and bounded from above RWCE on $G$ with ``slowly" changing edge-weights inherits the recurrence or transience of the initial weighted graph. In other words, we put quite a stringent assumption on the evolution of the environment of the RWCE in order to get a recurrence/transience result which holds for any graph. 
	
	The rest of the paper is organized as follows. In Section~\ref{mainresult}, we introduce our main result along with a few of its corollaries when applied to special cases. In Section~\ref{electrical}, we discuss some well-known results from electrical network theory which will be crucial in our proof. Section~\ref{single} is devoted to the proof of our main result.  Finally, Section~\ref{appn} contains some deferred proofs and useful lemmas. 
	
	\section{Main Result}{\label{mainresult}}
	Our condition on ``slow" changes is conveniently stated in terms of \textit{resistances}, which are simply the reciprocal of weights. For any $t\geq 0,$ we write $R_t:=1/C_t$. The term ``resistance" alludes to the well-studied connection between random walk on graphs and electrical networks. Indeed, our proof for the main result exploits that connection heavily, as was the case for techniques used by~\cite{Amir2020}. In the next section, we shall present a short summary of the tools borrowed from electrical network theory required for our analysis. For now, we state our main result.
	
	\begin{theorem}{\label{mainthm}}
		Let $G=(V,E)$ be any simple, undirected, locally finite and connected graph. 
		Let $\{\l X_t,G_t\r\}_{t\geq 0}$ be any proper and bounded from above RWCE on $G$. Assume that
		\begin{equation}{\label{condition}}
		\Gamma :=	\sum_{t \geq 0, \,e \in E}|R_t(e)-R_{t+1}(e)| < \infty
		\end{equation}
		almost surely. Then the RWCE $\{\l X_t,G_t\r\}_{t\geq 0}$ is recurrent (resp.~transient), provided the weighted random walk on $(G,C_0)$ is recurrent (resp.~transient).\footnote{For a random weight configuration $C_0$, we say that the weighted random walk on $(G,C_0)$ is recurrent/transient if the weighted random walk, conditional on $C_0$, is recurrent/transient almost surely.}
	\end{theorem}
	
	\begin{remark}{\label{boundbelow}}
			If the condition in~(\ref{condition}) is satisfied, note that almost surely $\sum_{e\in E}|R_t(e)-R_{t+1}(e)|$ is finite for each $t \geq 0$ and tends to $0$ as $t\to\infty.$ This is the reason we say that the weights are changing ``slowly." Since, $\sum_{e \in E} |R_t(e)-R_{t+1}(e)| < \infty$ almost surely for all $t \geq 0$, we know by Lemma~\ref{resistchange} that the weighted random walk on $(G,C_t)$ is recurrent (resp.~transient) for all $t \geq 0$, provided the walk on $(G,C_0)$ is recurrent (resp.~transient). 
	\end{remark}

	\begin{remark}{\label{below}}
		Under the condition in~(\ref{condition}), for any $e \in E$, we have $|R_t(e)-R_0(e)| \leq \Gamma$ and thus $C_t(e) \geq 1/(\Gamma + 1/C_0(e))$ for all $t \geq 0$. Hence, any proper RWCE satisfying~(\ref{condition}) must necessarily be bounded from below. As a result, any RWCE satisfying the hypotheses of Theorem~\ref{mainresult} is bounded (from above and below) and thus elliptic. 
		Condition~(\ref{condition}) also guarantees that $R_{\infty}(e) = \lim_{t \to \infty} R_t(e)$ exists almost surely for any $e \in E$. Moreover, $C_{\infty}(e) := 1/R_{\infty}(e)$ satisfies $C_{\infty}(e) \in (0, \infty)$ since the RWCE is bounded.
	\end{remark}

	 In the special case of monotone increasing (resp.~decreasing) RWCEs, it suffices for $\sum_{e}| R_{0}(e)-R_{\infty}(e)|$ to be finite, as mentioned in the following corollary.

	\begin{corollary}{\label{cor1}}
Let $G=(V, E)$ be simple, undirected, locally finite and connected.
Let $\left\{\left\langle X_{t}, G_{t}\right\rangle\right\}_{t \geq 0}$ be any proper, bounded from above RWCE on $G$ that is also edgewise monotone. Assume that the random variable $\sum_{e \in E}\left|R_{0}(e)-R_{\infty}(e)\right|$ is finite almost surely, where  $R_{\infty}(e):=\lim _{t \rightarrow \infty} R_{t}(e)$ for all $e \in E$. Then, $\left\{\l X_{t},G_t \r \right\}_{t \in \mathbb{N}}$ is recurrent (resp.~transient), provided the weighted random walk on $(G,C_0)$ is recurrent (resp.~transient).
	\end{corollary}
		\begin{proof}
		This follows immediately from Theorem~\ref{mainthm}. The edgewise limits exist almost surely by the monotonicity assumption, which also guarantees that 
		$$ \sum_{t \geq 0, \,e \in E}|R_t(e)-R_{t+1}(e)| = \sum_{e \in E}\left|R_{0}(e)-R_{\infty}(e)\right|.$$
	\end{proof}
	
	\begin{remark}
		The statement of Corollary~\ref{cor1} does not require the RWCE to be monotone increasing/decreasing. Indeed, the direction of monotonicity can be different for different edges.
	\end{remark}
	
	If we limit our interest to non-adaptive RWCEs on graphs with bounded degree, we can further alleviate the condition of summing over all edges in (\ref{condition}). The proof relies on the observation that the transition probabilities are ``local" functions of the weight configuration in the sense that they only depend on the weights of the edges adjacent to the current position of the walk. 
	
	\begin{corollary}
		Let $G=(V, E)$ be any simple, undirected, connected graph of bounded degree. Let $\left\{\left\langle X_{t}, G_{t}\right\rangle\right\}_{t \geq 0}$ be any proper, bounded from above and non-adaptive RWCE on $G$. Assume that  for each $e \in E$, we have $C_t(e) \to C_{\infty}(e)$ almost surely as $t \to \infty$ and the random variable
		$$ \sum_{t \geq 0} \sup_{e \in E} |R_t(e)-R_{\infty}(e)| $$
		is finite almost surely, where $R_{\infty}:=1/C_{\infty}$. Then, $\left\{\l X_{t},G_t \r \right\}_{t \geq 0}$ is recurrent (resp.~transient), provided the weighted random walk on $(G,C_{\infty})$ is recurrent (resp.~transient).
	
	\end{corollary}
	
	\begin{proof}
		Set $\mathscr{G}_{\infty} := \sigma \left(C_s : s \geq 0 \right).$ By Lemma~\ref{nonadaptive}, we can assume that for any $t \geq 0$ and $y \in V$, 
		\begin{equation}{\label{dynamX}}
			 \mathbb{P}\left[ X_{t+1}=y \mid \sigma \left(\mathscr{G}_{\infty}, \mathscr{F}_t\right) \right] = P(X_t,y;C_t).
			\end{equation}
		We now construct an auxiliary RWCE $\left\{\l X_t^{\prime}, G_t^{\prime} \r\right\}_{t \geq 0}$ such that $\left\{X_t^{\prime} \right\}_{t \geq 0}$ has the same distribution as $\left\{X_{t}\right\}_{t \geq 0}$. In particular, define $X_{0}^{\prime}:=X_{0}$ and for any $t \geq 0$, set $G_t^{\prime}=(V,E,C^{\prime}_t)$ where
			\begin{equation}{\label{defCprime}}
				C_{t}^{\prime}(e):= \begin{cases}C_{t}(e), & \text{ if }e \text { is adjacent to } X_{t}^{\prime}, \\ C_{\infty}(e), & \text{ if } e \text { is not adjacent to } X_{t}^{\prime},\end{cases}
			\end{equation}
		and for all $t \geq 0$ and $y \in V$, 
		$$ \mathbb{P} \left[X_{t+1}^{\prime} = y \mid \sigma \left( \mathscr{F}_{\infty}, \mathscr{F}_{t}^{\prime}\right) \right]=  P(X_{t}^{\prime},y; C_{t}^{\prime}),$$ 
		where $\mathscr{F}_{\infty} := \sigma \left( \l X_t,G_t \r : t \geq 0 \right)$ and $\mathscr{F}^{\prime}_t := \sigma \left( \l X_s^{\prime},G^{\prime}_s \r : 0 \leq s \leq t \right)$. It is obvious that $\left\{\l X^{\prime}_t, G^{\prime}_t \r\right\}_{t \geq 0}$ is an adaptive RWCE. Moreover, the definition in (\ref{defCprime}) guarantees that for all $t \geq 0$ and $y \in V$, 
		$$ \mathbb{P} \left[X_{t+1}^{\prime} = y \mid \sigma \left( \mathscr{F}_{\infty}, \mathscr{F}_{t}^{\prime}\right) \right]=  P(X_{t}^{\prime},y; C_{t}^{\prime}) = P(X_{t}^{\prime},y;C_{t}),$$ 
	and hence,
	$$ \mathbb{P} \left[X_{t+1}^{\prime} = y \mid \sigma \left( \mathscr{G}_{\infty}, X_0^{\prime}, \ldots, X^{\prime}_{t}\right) \right]= P(X_{t}^{\prime},y;C_{t}).$$
	This is same as the evolution dynamics for $\left\{X_t : t \geq 0\right\}$, conditioned on $\mathscr{G}_{\infty}$, as assumed in (\ref{dynamX}). Since, $X_0=X_0^{\prime}$, this shows that 
	$$ \mathcal{L} \left( \left\{X_t : t \geq 0\right\} \mid \mathscr{G}_{\infty} \right) = \mathcal{L} \left( \left\{X_t^{\prime} : t \geq 0\right\} \mid \mathscr{G}_{\infty} \right), $$
		and thus  $\left\{X_{t}^{\prime}\right\}_{t \geq 0}=\left\{X_{t}\right\}_{t \geq 0}$ in distribution. Therefore, it is enough to consider recurrence/transience of the RWCE  $\left\{\l X_t^{\prime}, G_t^{\prime} \r\right\}_{t \geq 0}$. We shall show that it satisfies the conditions of Theorem~\ref{mainthm}. 
		
		Let $d_{\text{max}}$ be the maximum degree of vertices in $G$, $R_t^{\prime}:=1/C_t^{\prime}$ and $\Delta_t := \sup_{e \in E} |R_t(e)-R_{\infty}(e)|$. Since, $\sum_{t \geq 0} \Delta_t < \infty$ almost surely, we have $R_{\infty}(e)< \infty$, i.e., $C_{\infty}(e)>0$ almost surely for all $e \in E$. This observation, along with the hypothesis that the walk $\left\{\l X_t,G_t \r \right\}_{t \geq 0}$ is proper, implies that $C^{\prime}_t(e)>0$ almost surely for all $e \in E$. Moreover,
		$$ C^{\prime}_t(e) \leq \max(C_t(e), C_{\infty}(e)) \leq \sup_{t \geq 0} C_t(e) < \infty,$$
		almost surely, since the original RWCE is bounded from above. 
		Finally, note that for any $e \in E$, we have	
		$$\left|R_{t}^{\prime}(e)-R_{t+1}^{\prime}(e)\right| \leq \left|R_{t}^{\prime}(e)-R_{\infty}(e)\right|  + \left|R_{t+1}^{\prime}(e)-R_{\infty}(e)\right|  \leq \Delta_{t}+\Delta_{t+1}.$$
		 Also, $R_{t}^{\prime}$ and $R_{t+1}^{\prime}$ differ on at most $2d_{\text{max}}$ edges which are adjacent to $X_t^{\prime}$ or $X^{\prime}_{t+1}$. Thus,
		$$
		\sum_{t \geq 0} \sum_{e \in E} \left|R_{t}^{\prime}(e)-R_{t+1}^{\prime}(e)\right| \leq 2 d_{\text{max}} \sum_{t \geq 0}\left(\Delta_{t}+\Delta_{t+1}\right) \leq 4 d_{\text{max}} \sum_{t \geq 0} \Delta_{t}.
		$$
		As the right-most term is finite almost surely, we conclude by Theorem~\ref{mainthm} that the RWCE $\left\{\l X_t^{\prime}, G_t^{\prime} \r\right\}_{t \geq 0}$ is recurrent (resp.~transient) if the weighted walk on $(G, C_0^{\prime})$ is recurrent (resp.~transient). Since, $C_{\infty}$ and $C^{\prime}_0$ differ on at most $d_{\text{max}}$ many edges, by Lemma~\ref{resistchange} the walk on  $(G, C_0^{\prime})$ is recurrent (resp.~transient) if and only if the walk on $(G, C_{\infty})$ is recurrent (resp.~transient). This completes the proof.
	\end{proof}
	
	The rest of the paper is devoted to the proof of Theorem~\ref{mainthm}. We end this section with a short overview of our proof. Let $\{\l X_t,G_t\r\}_{t \geq 0}$ be any RWCE on $G=(V,E)$ satisfying the conditions of \Cref{mainthm}. Fix any arbitrary vertex $a\in V$. To show recurrence (resp.~transience), it suffices by ellipticity to show that $a$ is almost surely visited infinitely (resp.~finitely) often. 
	To show that $a$ is almost surely visited infinitely (resp.~finitely) often, we use the standard technique  of constructing a super/sub-martingale of the form $\{f_t(X_t)\}_{t\geq 0}$ and then using the optional stopping theorem to bound the probability of return to $a$. This part of the proof borrows heavily from the tools of electrical network theory, a short summary of those tools is presented in Section~\ref{electrical}. The martingale technique that we use, presented in Section~\ref{single}, is an extension of the methods used in \cite{Amir2020}.

	\section{Electrical Network Theory}{\label{electrical}}
	
	In this section, we recall some important results on the characterization of recurrence or transience given by electrical network theory. We heavily follow \cite[Chapter 1]{Grimmett2018}, which includes detailed proofs and further material for the interested reader. Other standard references include \cite{Doyle1984, Lyons2017}. We begin with some notation. Fix any $G=(V, E)$ satisfying the usual assumptions (i.e., simple, undirected, locally finite and connected) and take any $C: E \rightarrow(0, \infty)$. Fix any $a \in V$, which we consider as the origin of $G$. For $n \geq 0$, let $V_{(n)}:=\{x \in V: d(a, x) \leq n\}$ and $\partial V_{(n)}:=\{x \in V: d(a, x)=n\}$, where $d$ is the shortest-path distance on $G$. Let $G_{(n)}=(V_{(n)},E_{(n)})$ be the finite subgraph of $G$ induced by $V_{(n)}$. Finally, let $\left\{X_{t}\right\}_{t=0}^{\infty}$ denote the weighted random walk on $(G, C)$. For this section, $\mathbb{P}_x$ will denote the law of $\left\{X_t : t \geq 0\right\}$, conditioned on $X_0=x \in V.$ For any $S \subseteq V$, let $ \tau_S := \inf \left\{t \geq 0 : X_t \in S\right\}$. When $S=\left\{x\right\}$ is a singleton, we shall write $\tau_x$ instead of $\tau_{\left\{x\right\}}$.  
	
	\begin{result}{\label{res1}}
	For any $x\in V_{(n)}$, let $v(x):= \mathbb{P}_x\left[\tau_{a} < \tau_{\partial V_{(n)}} \right]$ denote the probability that the walk, starting from $x$, visits the origin $a$ before reaching $\partial V_{(n)}$. Then, $v : V_{(n)}\to [0,1]$ is the unique harmonic function on $V_{(n)} \setminus \left(\{a\} \cup \partial V_{(n)}\right)$ satisfying the boundary conditions $v(a)=1$ and $v(x)=0$ for all $x \in \partial V_{(n)}$. Thus, $v(x)$ equals the voltage at $x$ when viewing $(G_{(n)},C)$ as an electrical network with $a \in V$ as the source (kept at voltage $1$), $\partial V_{(n)}$ as the sink (grounded/ kept at voltage $0$) and resistances $R:=1/C$. 
	\end{result}
	
	We further elaborate on the basic theory of finite electrical networks. For this part alone, let $G=(V, E)$ be finite. Fix $a \in V$ as the source and $b \in V \backslash\{a\}$ as the sink. The central objects of electrical network theory are $a\to b$ flows which are defined below.

\begin{definition}
	We say that $j: V^{2} \rightarrow \R$ is an $a \to b$ flow on $(G,C)$ if for all $\left\{x,y\right\} \in E$, we have $j(x, y) =-j(x, y); \; j(x, y)=0$ if $\left\{x,y\right\} \notin E;$ $J_{x}:=\sum_{y \sim x} j(x, y)=0$ for any $x \neq a,b$; $J_a \geq 0$.  
	One can check that $J_{b}=-J_{a}$ and we define $|j|:=J_{a}$ as the size of the flow $j$. We say that $j$ is a unit flow if $|j|=1$. Moreover, we assign an energy to the flow $j$, given by
	$$
	E(j):=\sum_{e \in E} j^{2}(e) R(e)
	$$
	where $R(e):=1/C(e)$ is the resistance of $e \in E$. 
\end{definition}	

\begin{result}{\label{res2}}
	Among all possible unit $a\to b$ flows on $(G, C)$, there is a unique flow $i$ which obtains the minimum possible energy. This minimizing flow is the unit current on the network $(G,C)$ with source $a$ and sink $b$. In other words, the minimizing flow $i$ satisfies the Kirchhoff potential law: Namely, we have $\sum_{k=1}^n i(x_k,x_{k+1})R(x_k,x_{k+1}) =0$ for any cycle $x_1, \ldots, x_n,x_{n+1}=x_1$ in $V$. This variational property is known as Thompson's principle.
	
	 Moreover, for any $\left\{x,y\right\}\in E$, the current flowing from $x$ to $y$ has the following probabilistic interpretation: It is the expected number of net crossings of $\left\{x,y\right\}$, in the direction from $x$ to $y$, by the walk $\left\{X_t : t \geq 0\right\}$ starting from $a$ before it reaches $b$. In other words,
	$$ i(x,y) = \mathbb{E}_a \sum_{t=1}^{\tau_b} \mathbbm{1}_{(X_{t-1}=x,X_t=y)} -  \mathbb{E}_a \sum_{t=1}^{\tau_b} \mathbbm{1}_{(X_{t-1}=y,X_t=x)}.$$
	We also have $|i(x,y)| \leq 1$ for all $x,y$. 

\end{result}

\begin{remark}
	The energy of the unit current, $\mathcal{R}(a, b):=\sum_{e \in E} i^{2}(e) R(e)$, is the effective resistance between $a$ and $b$ in $(G, C)$. Intuitively, the network $(G, C)$ can be reduced to a single resistor between $a$ and $b$ with resistance $\mathcal{R}(a, b)$. Also, we remark that the voltage difference between $a$ and  $b$ induced by $i$ will exactly equal $\mathcal{R}(a, b)$. In other words, the rescaled current $i/\mathcal{R}(a,b)$ induces the voltage $v$ on $V$ with $v(a)=1$ and $v(b)=0$. \textit{Ohm's Law} then implies that for any $x,y \in V$, we have $v(x)-v(y) = i(x,y)R(x,y)/\mathcal{R}(a,b).$
\end{remark}	
	
We are now ready to describe the connection between recurrence or transience and electrical network theory. Let $\mathcal{R}_{n}$ denote the effective resistance between $a$ and $\partial V_{(n)}$ in $\left(G_{(n)}, C\right)$. Indeed, $\partial V_{(n)}$ may not be a single vertex, and in this case we identify all vertices in $\partial V_{(n)}$ into a single vertex while removing any self-loops that occur: Let $\widetilde{E}_{(n)}$ be the edge-set after performing this process (in other words, $\widetilde{E}_{(n)}$ contains all edges in $E$ with at least one end-point in $V_{(n-1)}$). An extremely useful property of effective resistance is given by \textit{Rayleigh's monotonicity principle}, which states that the effective resistance of a network can only increase if we increase edge-resistances. It follows that $\left\{\mathcal{R}_{n}\right\}_{n=0}^{\infty}$ is non-decreasing, hence the limit $\mathcal{R}_{\infty}:=\lim _{n \rightarrow \infty} \mathcal{R}_{n}$ always exists. 

\begin{result}{\label{res3}}
For the setting described above,
$$ \mathbb{P}_a \left[ X_t =a \text{ for some } t \geq 1 \right] = 1 - \dfrac{1}{C(a)\mathcal{R}_{\infty}},$$
where $C(a):=\sum_{x \sim a} C(a,x)$. In particular, the weighted random walk on $(G,C)$ is recurrent if $\mathcal{R}_{\infty}=\infty$ and transient if $\mathcal{R}_{\infty}< \infty$. 
\end{result}

\begin{lemma}{\label{resistchange}}
	Consider two weight configurations $C_1$ and $C_2$ on $G=(V,E)$ satisfying the usual conditions, i.e., simple, undirected, locally finite and connected. Assume that $\sum_{e \in E} |R_1(e)-R_2(e)| < \infty$ where $R_i=1/C_i$ for $i=1,2$. Then the weighted random walk on $(G,C_1)$ is recurrent/transient if and only if the same is true for the walk on $(G,C_2)$.   
	\end{lemma}
	
	\begin{proof}
		Fix $a \in V$ as the origin and let $b_n$ denote the vertex obtained by identifying vertices in $\partial V_{(n)}$.  Let $\mathcal{R}_{i,n}$ be the effective resistance between $a$ and $b_n$ when the weight configuration is $C_i$, with $\mathcal{R}_{i, \infty} := \lim_{n \to \infty} \mathcal{R}_{i,n}$, for $i=1,2$. In light of Result~\ref{res3}, it is enough to show that $\lim_{n \to \infty} \left(\mathcal{R}_{1,n} - \mathcal{R}_{2, n} \right)$ is finite. Let 
		$$\mathcal{J}(a \to b_n) := \left\{ j \; : \; j \text{ is a unit } a \to b_n \text{ flow with } |j(x,y)| \leq 1, \text{ for all } \left\{x,y\right\} \in \widetilde{E}_{(n)}\right\}.$$
		 Then, Result~\ref{res2} implies that 
		\begin{align*}
			\rvert \mathcal{R}_{1,n} - \mathcal{R}_{2,n} \rvert &= \Bigg \rvert \inf_{j \in \mathcal{J}(a \to b_n)} \sum_{e \in \widetilde{E}_{(n)}} j^2(e)R_1(e) -  \inf_{j \in \mathcal{J}(a \to b_n)} \sum_{e \in \widetilde{E}_{(n)}} j^2(e)R_2(e) \Bigg \rvert \\
			& \leq \sup_{j \in \mathcal{J}(a \to b_n)} \sum_{e \in \widetilde{E}_{(n)}} j^2(e) |R_1(e)-R_2(e)| \leq \sum_{e \in E} |R_1(e)-R_2(e)| < \infty.
		\end{align*}
		Taking $n \to \infty$, we complete the proof.
	\end{proof}

	\section{Proof of the Main Result}{\label{single}}
 
As mentioned before, we will construct a super/sub-martingale and then apply the optional stopping theorem to derive  conditions for infinite/finite returns to the vertex $a$. Then, we will show that these conditions are satisfied assuming the condition of \Cref{mainthm}.
	
	\subsection{Super/sub-martingales}
	We first construct the desired super/sub-martingale. Fix the origin $a \in V$ and assume that the RWCE $\{\l X_t,G_t\r\}_{t\geq 0}$ is given. 
	For $n\geq 1, t\geq 0,$ we construct an electrical network on the graph $G_{(n)}$ as follows: For any $e \in E_{(n)}$, we set the conductance of $e$ as $C_t(e)$. We keep the source $a$ at voltage $1$ and $\partial V_{(n)}$ at voltage $0.$ For all $x\in V_{(n)},$ let $v_{n,t}(x)$ denote the (random) voltage at $x$ in the above constructed electrical network on $(G_{(n)},C_t)$. If $x\notin V_{(n)},$ we define $v_{n,t}(x):=0.$  
	
	According to Result~\ref{res1}, $v_{n,t}(x)$ equals the probability that the weighted random walk on $(G_{(n)},C_t)$, starting from $x$, will visit $a$ before hitting $\partial V_{(n)}$. In particular, the function $v_{n,t}$ is harmonic on $V_{(n-1)} \setminus \left\{a\right\}$, i.e.,
	\begin{equation}{\label{harmonic}}
		v_{n,t}(x) = \sum_{y: y \sim x} \dfrac{C_t(x,y)}{\sum_{z : z \sim x} C_t(x,z)} v_{n,t}(y), \; \; \forall \; x \in V_{(n-1)} \setminus \left\{a\right\}.
	\end{equation}
	Moreover, since $C_t$ is proper and $G$ is connected and locally finite, one can easily see that the weighted random walk on $(G_{(n)},C_t)$, starting from $x \in V_{(n-1)} \setminus \left\{a\right\}$, has positive probability to visit $a$ before hitting $\partial V_{(n)}$. In other words, $v_{n,t}(x) \in (0,1]$ for $x \in V_{(n-1)} \setminus \left\{a\right\}$. 
	
	Now, we define $\tau_n:=\inf\{t\geq 0: X_t \notin V_{(n-1)} \setminus \left\{a\right\}\}$ and recall that $\mathscr{F}_t=\sigma \left( \l X_s, G_s \r : s =0, \ldots, t \right)$ for all $t \geq 0$. The harmonicity of the function $v_{n,t}$ guarantees that 
	\begin{equation}{\label{harmonic2}}
		 \mathbb{E} \left[ v_{n,t}\left( X_{\tau_n \wedge (t+1)} \right) \Big \rvert \mathscr{F}_t \right] = v_{n,t}(X_{\tau_n \wedge t}), \; \forall \; t \geq 0.
\end{equation}
	 Unfortunately, for arbitrary $x\in V$, the sequence $\{v_{n,t}(x)\}_{t\geq 0}$ is not necessarily monotone and thus the process $\{v_{n,t}(X_{\tau_n \wedge t})\}_{t\geq 0}$ is not a super/sub-martingale.
	
	To bypass this difficulty and facilitate a more general class of super/sub-martingales, we consider the following approach. Take $G=(V,E)$ with origin $a$ and remove the vertex $a$ along with its adjacent edges. Let $V^{\prime}$ be any connected component of the resulting graph. Since $G$ is connected and locally-finite, there are finitely many components resulting from the removal of $a$. Moreover, regarding recurrence or transience, we will see that only the infinite components are of interest. Thus, assume that $V'$ is countably infinite and set $V^*= V^{\prime} \cup \left\{a\right\}$. 

	Let $G^* = \left( V^*,E^*\right)$ be the subgraph of $G$ induced by $V^*$. It is clear that $G^*$ is connected and for any $x,y \in V^* \setminus \left\{a\right\}$, there is a path in $G^*$ connecting $x$ and $y$ which does not visit $a$. Also, any path connecting $x \in V^*$ to $y \in V \setminus V^*$ must visit the vertex $a$.  For any $n \geq 0$, let 
	$$ V^{*}_{(n)} := \left\{ x \in V^{*} : d(x,a) \leq n \right\} = V^{*} \cap V_{(n)}, \;\; \partial V^{*}_{(n)} := \left\{ x \in V^{*} : d(x,a) = n \right\} = V^{*} \cap \partial V_{(n)}.$$
    For any time $t \geq 0$, we can construct an electrical network on $G^*$ with conductance configuration given by $C_t^*:= C_t \big \rvert_{E^*}$, keeping $a$ at voltage $1$ and the vertices in $\partial V_{(n)}^{*}$ grounded. If $v_{n,t}^{*}(x)$ is the voltage at vertex $x \in V_{(n)}^{*}$ in this network, then $v_{n,t}^{*}(x)$ is the probability that the weighted random walk on $(G^*, C_t^*)$, starting from $x$, will visit $a$ before hitting $\partial V_{(n)}^{*}$. Since any path in $G$ from $x$ to $V \setminus V^*$ has to visit $a$, it is obvious that $v_{n,t}^{*}(x)$ equals the probability that the weighted random walk on $\left(G_{(n)}, C_t \right)$, starting from $x$, will visit $a$ before hitting $\partial V_{(n)}$; in other words $v_{n,t}^{*}(x) = v_{n,t}(x)$ for all $x \in V_{(n)}^{*}$. Set,
	\begin{align*}
			\alpha_{n,t}^*:&=\sup_{x\in V^* \setminus\{a\}}\frac{1-v_{n,t+1}(x)}{1-v_{n,t}(x)}\geq 1,\\
		\beta_{n,t}^{*}:&=\inf_{x\in V^{*} \setminus\{a\}}\frac{1-v_{n,t+1}(x)}{1-v_{n,t}(x)}\leq 1,
	\end{align*}
	for all $n \geq 1, t \geq 0$, 
	where the inequalities follow by considering $x\in \partial V_{(n)}^*.$ Also, the quantities are well-defined (i.e., positive and finite) since $V^*$ is countably infinite and thus $v_{n,t}(x) \in (0,1)$ for any $x \neq a$ and $t \geq 0$. Note that $\alpha^*_{n,t}$ and $\beta^*_{n,t}$ are $\mathscr{F}_{t+1}$-measurable random variables.

    \begin{lemma}{\label{supersub}}
		Fix $n\geq 1$ and let 
		\begin{align*}
			A_t^{*(n)} &= \frac{1-v_{n,t}(X_t)}{\prod_{s=0}^{t-1}\alpha^*_{n,s}},\hspace{.2cm}
			B_t^{*(n)} = \frac{1-v_{n,t}(X_t)}{\prod_{s=0}^{t-1}\beta^*_{n,s}}
		\end{align*}
		for $t\geq 0.$ Then, $\{A^{*(n)}_{\tau_n \wedge t}\}_{t\geq 0}$ is a supermartingale and $\{B^{*(n)}_{\tau_n \wedge t}\}_{t\geq 0}$ is a submartingale with respect to the filtration $\{\mathscr{F}_t\}_{t\geq 0}$, provided $X_0 \in V^*$.
	\end{lemma}
	
	\begin{proof}
	We start by observing that both $A^{*(n)}_t$ and $B^{*(n)}_t$ are $\mathscr{F}_t$-measurable while $0 \leq A_t^{*(n)} \leq 1$ and $B^{*(n)}_t \geq 0$, hence taking conditional expectations is well-defined. It now suffices to prove only the supermartingale case as the rest of the proof for the submartingale case is identical.   Firstly, note that, on the event $(X_{t+1} \in V^* \setminus \left\{a\right\})$, we have
		\begin{align*}
			A^{*(n)}_{t+1} &= \frac{1-v_{n,t+1}(X_{t+1})}{\prod_{s=0}^{t}\alpha^*_{n,s}} = \frac{1-v_{n,t}(X_{t+1})}{\prod_{s=0}^{t-1}\alpha^*_{n,s}} \dfrac{1-v_{n,t+1}(X_{t+1})}{(1-v_{n,t}(X_{t+1})) \alpha^*_{n,t}} \leq \frac{1-v_{n,t}(X_{t+1})}{\prod_{s=0}^{t-1}\alpha^*_{n,s}},
		\end{align*}
		whereas the inequality is trivially true on $(X_{t+1}=a)$. Hence, on $(X_0 \in V^*, \tau_n >t)$, 
		\begin{align*}
			\E\left[A^{*(n)}_{\tau_n \wedge (t+1)}\Big \rvert \mathscr{F}_t\right]&\leq \E\left[\frac{1-v_{n,t}(X_{t+1})}{\prod_{s=0}^{t-1}\alpha^*_{n,s}}\Bigg \rvert \mathscr{F}_t\right]=\frac{1}{\prod_{s=0}^{t-1}\alpha^*_{n,s}}\E\left[1-v_{n,t}(X_{t+1})\big \rvert \mathscr{F}_t\right] =A^{*(n)}_{\tau_n \wedge t},
		\end{align*}
	where the last equality follows from (\ref{harmonic2}). Finally, on the event $(X_0 \in V^*, \tau_n \leq t)$, 
		\begin{align*}
			\E \left[A^{*(n)}_{\tau_n \wedge (t+1)}\Big \rvert  \mathscr{F}_t\right] = \E \left[A^{*(n)}_{\tau_n} \Big \rvert \mathscr{F}_t\right] = A_{\tau_n}^{*(n)} = A_{\tau_n \wedge t}^{*(n)}
		\end{align*}
		as desired and we conclude our proof.
	\end{proof}

	\subsection{Optional Stopping Theorem}
	We now intend to apply the \textit{Optional Stopping Theorem} to the super/sub-martingale constructed above. For the results of this section, we remark that it suffices to only assume that the given RWCE is proper and elliptic. 
	We begin with the supermartingale $\{A_{\tau_n \wedge t}^{*(n)}\}_{t\geq 0}$, continuing with the notations introduced before the statement of Lemma~\ref{supersub}.
	\begin{lemma}{\label{ost1}}
		Let $\{\l X_t,G_t\r\}_{t\geq 0}$ be any proper and elliptic RWCE on $G=(V,E)$. 
		Assume that there exists $N \in \mathbb{N}$ such that 
		$$ \sup_{n \geq N} \prod_{t \geq 0} \alpha^*_{n,t} < \infty,$$
		almost surely and $v_{n,t}(x)\to 1$ almost surely as $n\to\infty$ for any $x\in V$ and $t \geq 0$. Then for any $T \geq 0$, 
		$$ \mathbb{P}\left(X_t=a \text{ for some } t \geq T \mid \mathscr{F}_T\right) =1,$$
		almost surely on the event $(X_T \in V^*)$.
	\end{lemma}
	
	\begin{proof}
		Since $\alpha^*_{n,t}\geq 1$, note that the conditions of the lemma hold true for any sub-process $\{\l X_t,G_t\r\}_{t\geq T}$ where $T>0.$ Hence, it suffices to show that almost surely on the event $(X_0 \in V^*)$, 
		$$\mathbb{P}\left(X_t=a \text{ for some } t \geq 0 \mid \mathscr{F}_0\right) =1.$$
  Fix $n\geq N$, $C \in (0,\infty)$ and recall the supermartingale $\{A_{\tau_n \wedge t}^{*(n)}\}_{t\geq 0}$ from Lemma~\ref{supersub}. Since $0 \leq  A^{*(n)}_{\tau_n \wedge t}\leq 1$ for all $t\geq 0$ and $\tau_n < \infty $ almost surely (by ellipticity of the RWCE and finiteness of $V_{(n-1)}$), the optional stopping theorem gives $\E[A_{\tau_n}^{*(n)} \mid \mathscr{F}_0] \mathbbm{1}_{(X_0 \in V^*)}\leq A^{*(n)}_0 \mathbbm{1}_{(X_0 \in V^*)}.$ Hence, on the event $(X_0 \in V^*)$, we have the following almost surely:
		\begin{align*}
			1-v_{n,0}(X_0) = A^{*(n)}_0\geq \E\left[\frac{1-v_{n,\tau_n}(X_{\tau_n})}{\prod_{t=0}^{\tau_n-1} \alpha^*_{n,t}} \Bigg \rvert \mathscr{F}_0\right] &\geq \E\left[\frac{1-v_{n,\tau_n}(X_{\tau_n})}{\prod_{t=0}^{\tau_n-1} \alpha^*_{n,t}}, \sup_{m \geq N} \prod_{t \geq 0} \alpha^*_{m,t} \leq C  \Bigg \rvert \mathscr{F}_0 \right]  \\
			& \geq \dfrac{1}{C} \mathbb{E} \left[ X_{\tau_n} \in \partial V_{(n)}, \sup_{m \geq N} \prod_{t \geq 0} \alpha^*_{m,t} \leq C  \Bigg \rvert \mathscr{F}_0\right],
		\end{align*}
		which can be rearranged as 
		$$ \mathbb{E} \left[ X_{\tau_n} \in \partial V_{(n)}, \sup_{m \geq N} \prod_{t \geq 0} \alpha^*_{m,t} \leq C  \Bigg \rvert \mathscr{F}_0\right] \leq C 	(1-v_{n,0}(X_0)).$$
		Since, $v_{n,0}(X_0)$ converges almost surely to $1$, the right hand side of the above converges to $0$ for any fixed $C \in (0,\infty)$, as we take $n \to \infty$. Now we take $C \to \infty$ and use the hypothesis of the lemma to conclude that $\lim_{n \to \infty} \mathbb{P}(X_{\tau_n} \in \partial V_{(n)}\mid \mathscr{F}_0) =0$, or equivalently  $\lim_{n \to \infty} \mathbb{P}(X_{\tau_n} =a \mid \mathscr{F}_0) =1,$ almost surely on $(X_0 \in V^*)$. This completes the proof.
	\end{proof}
	Next, we proceed similarly with the submartingale $\{B^{*(n)}_{\tau_n \wedge t}\}_{t\geq 0}.$

		\begin{lemma}{\label{ost2}}
		Let $\{\l X_t,G_t\r\}_{t\geq 0}$ be any proper and elliptic RWCE on $G=(V,E)$. Assume that there exists a constant $\delta >0$, $N \in \mathbb{N}$ and $A \in \mathscr{F}_0$ such that 
		$$ \inf_{n \geq N} \prod_{t \geq 0} \beta^{*}_{n,t} \geq \delta, \; \text{ almost surely on the event }A.$$
		Then for any $T \geq 0$, we have the following on the event $A \cap (X_T \in V^*\setminus \{a\})$ : 
		$$ \P\left( X_t \neq a, \; \forall \; t \geq T \mid \mathscr{F}_T\right)  \geq \delta \inf_{x \in \partial V_{(1)}^{*}, \; n \geq N} (1-v_{n,T}(x)), \;\; \text{ almost surely}.$$
	\end{lemma}	
	\begin{proof}
	Similar to the proof of Lemma~\ref{ost1}, it is enough to consider $T=0$. Introduce the notation 
		$$ V^{\text{sup}}_t := \sup_{x \in \partial V_{(1)}^*, n \geq N} v_{n,t}(x) = \sup_{x \in V^* \setminus \left\{a\right\}, n \geq N} v_{n,t}(x),$$
		where the second equality in the above equation follows from Lemma~\ref{flowdirec}. Fix some $n\geq N$ and recall the submartingale $\{B^{*(n)}_{\tau_n \wedge t}\}_{t\geq 0}$ from Lemma~\ref{supersub}. Under the hypothesis for this lemma, on the event $A$, we have $0 \leq B_{\tau_n \wedge t}^{*(n)}\leq 1/\delta<\infty,$ and hence the optional stopping theorem gives $\E[B_{\tau_n}^{*(n)} \mid \mathscr{F}_0]\geq B^{*(n)}_0$.
		Thus, we have almost surely on $A$, 
		\begin{align*}
			1-v_{n,0}(X_0) =B^{*(n)}_0 \leq \E\left[\frac{1-v_{n,\tau_n}(X_{\tau_n})}{\prod_{t=0}^{\tau_n-1} \beta^*_{n,t}} \Bigg \rvert \mathscr{F}_0 \right]\leq \frac{\P[X_{\tau_n} \in \partial V_{(n)}\mid \mathscr{F}_0]}{\delta},
		\end{align*}
		which can be rearranged as 
		$\P[X_{\tau_n} \in \partial V_{(n)}\mid \mathscr{F}_0]\geq \delta(1- v_{n,0}(X_0))$.  Taking $n \to \infty$, we obtain 
     $$\P \left[X_t \neq a, \; \forall \; t \geq 0 \mid \mathscr{F}_0\right]\geq \delta(1-V^{\text{sup}}_0)$$ on the event $A \cap (X_0 \in V^*\setminus\{a\})$.
	\end{proof}
	

	\subsection{Bounding Voltage Ratios}
	Having Lemma~\ref{ost1} and Lemma~\ref{ost2} at our disposal, we want to use these results to prove Theorem~\ref{mainthm}. For this purpose, we estimate $\alpha^*_{n,t}$ and $\beta^*_{n,t}$ by deriving an upper bound for $\left|(1-v_{n,t+1}(x))/(1-v_{n,t}(x))-1\right|$. We begin with the following expression for $|v_{n,t}(x)-v_{n,t+1}(x)|.$
	
	\begin{lemma}{\label{diffes}}
		For any $n\geq 1,$ $t\geq 0,$ and $x\in V_{(n-1)}\setminus\{a\}$ we have
		\begin{align*}
			v_{n,t}(x)-v_{n,t+1}(x) &= \frac{1}{\mathcal{R}_{n,t}}
			\sum_{e=\{y,z\}\in E_{(n)}}(R_{t}(e)-R_{t+1}(e))\cdot i_{x,\{a\}\cup\partial V_{(n)}}^{n,t+1}(y,z)\cdot i_{a,\partial V_{(n)}}^{n,t}(y,z)
		\end{align*}
		where $\mathcal{R}_{n,t}$ is the effective resistance between $a$ and $\partial V_{(n)}$ in $(G_{(n)},C_t)$. Also, $i^{n,t}_{y,S}$ is the unit current in $(G_{(n)},C_t)$ when $y$ is the source and $S\subseteq V_{(n)}\setminus\{y\}$ is kept grounded; $i^{n,t+1}$ is also defined in a similar manner. 
	\end{lemma}
	
	\begin{proof}
		Note that all random variables in the claim are determined given $C_t$ and $C_{t+1}.$ The key idea is to represent $v_{n,t+1}(x)$ in terms of the current $i_1:=i_{x,\left\{a\right\}\cup\partial V_{(n)}}^{n,t+1}.$ Namely, we claim that 
		\begin{align}
		1-	v_{n,t+1}(x) = \sum_{z\in\partial V_{(n)}}\sum_{y\in V_{(n)}}i_1(y,z).
			\label{eq:identity}
		\end{align}
		In words, the right-hand side 
		of \eqref{eq:identity} is the total amount of current in $i_1$ that flows into $\partial V_{(n)}.$ Recall that the probabilistic interpretation of $i_1(y,z)$ as mentioned in Result~\ref{res2}. 
		Namely, we perform a weighted random walk on $(G_{(n)},C_{t+1})$ starting from $x$ and terminate it once it reaches $a$ or $\partial V_{(n)}$. By Result~\ref{res2}, $i_1(y,z)$ equals the expected net number of crossings of $\left\{y,z\right\}$ in the direction from $y$ to $z$  by the walk. In particular, it is zero if $y\nsim z.$ Taking $y\sim z$ as specified in the summation above, if $y\in \left\{a\right\}\cup\partial V_{(n)}$ we also have $i_1(y,z)=0$ as $\{y,z\}$ is never crossed. On the other hand, if $y\in V_{(n-1)}\setminus\{a\}$, we can cross $\{y,z\}$ exactly once in the direction $(y,z)$ as the walk will terminate after the crossing and we can't cross $\left\{y,z\right\}$ in the direction $(z,y)$; therefore,  $i_1(y,z)$ equals the probability that the walk terminates after crossing $\{y,z\}.$ Since, the walk can cross only one edge of the form $\left\{y,z\right\}$ with $z \in \partial V_{(n)}$ and $y \in V_{(n-1)}\setminus \left\{a\right\},$ we can conclude that  the right-hand side of \eqref{eq:identity} is simply the probability that the weighted random walk on $(G_{(n)},C_{t+1})$ beginning at $x$ will terminate after hitting $\partial V_{(n)}$, i.e., it will hit $\partial V_{(n)}$ before reaching $a.$ By the probabilistic interpretation of voltage, this is exactly $1-v_{n,t+1}(x).$ 
		
		The rest of our proof is routine algebra of flows, which we explain below. First, using the definition of flows,  we observe that $\sum_{y \in V_{(n)}} i_1(y,z) = 0$ for all $z \in V_{(n-1)} \setminus \left\{a,x\right\}$, whereas $\sum_{y \in V_{(n)}} i_1(y,x) = -1.$ Thus $(1-v_{n,t}(z)) \sum_{y \in V_{(n)}} i_1(y,z) = 0$, for all $z \in V_{(n-1)} \setminus \left\{x\right\}$. We can, therefore, extend \eqref{eq:identity} to get
		\begin{align*}
			v_{n,t}(x)-v_{n,t+1}(x)&= (1-v_{n,t+1}(x)) - (1-v_{n,t}(x)) \\
			& =  \sum_{z\in\partial V_{(n)}}\sum_{y\in V_{(n)}}i_1(y,z) + (1-v_{n,t}(x)) \sum_{y \in V_{(n)}} i_1(y,x) \\
			&= \sum_{z\in\partial V_{(n)}}(1-v_{n,t}(z))\sum_{y\in V_{(n)}}i_1(y,z) + (1-v_{n,t}(x)) \sum_{y \in V_{(n)}} i_1(y,x)\\
			& = \sum_{z\in\partial V_{(n)}\cup \left\{x\right\}}(1-v_{n,t}(z))\sum_{y\in V_{(n)}}i_1(y,z) \\
			&= \sum_{z\in V_{(n)}}(1-v_{n,t}(z))\sum_{y\in V_{(n)}}i_{1}(y,z).
		\end{align*}
		As current is antisymmetric, we further obtain
		\begin{align*}
			v_{n,t}(x)-v_{n,t+1}(x)&=\frac{1}{2}\sum_{y,z\in V_{(n)}}(v_{n,t}(y)-v_{n,t}(z))\cdot i_1(y,z)\\
			&= \frac{1}{\mathcal{R}_{n,t}}\sum_{e=\{y,z\}\in E_{(n)}}R_{t}(e)\cdot i_0(y,z)\cdot i_1(y,z),
		\end{align*}
		where $i_0:=i_{a,\partial V_{(n)}}^{n,t}$ and the second equality follows from \textit{Ohm's law}. 
		
		To conclude, it suffices to show that
		\begin{align*}
			L:=\sum_{e=\{y,z\}\in E_{(n)}}R_{t+1}(e)\cdot i_{1}(y,z)\cdot i_{0}(y,z)=0.
		\end{align*}
		We evaluate $L$ by essentially reversing the above process. Let $\phi(x)$ denote the voltage at $y \in V_{(n)}$ induced by $i_1$. Then, by \textit{Ohm's law} we have
		\begin{align*}
			L&=\sum_{e=\{y,z\}\in E_{(n)}}(\phi(y)-\phi(z))\cdot i_{0}(y,z)\\
			&=\frac{1}{2}\sum_{y,z\in V_{(n)}}(\phi(y)-\phi(z))\cdot i_{0}(y,z)\\
			&=\sum_{y\in V_{(n)}}\phi(y)\sum_{z\in V_{(n)}}i_{0}(y,z),
		\end{align*}
		where the second and third equalities follow since current is antisymmetric. Again we have, $\sum_{z \in V_{(n)}} i_0(y,z) = 0$ for any $y \in V_{(n-1)} \setminus \left\{a\right\}$, whereas $\phi(y)=0$ for $y \in \left\{a\right\} \cup \partial V_{(n)}$ (by definition of $\phi$ and $i_1$). This shows that $L=0$ and completes the proof.
	
	\end{proof}
	
	Our next target is to compute an appropriate lower bound for $(1-v_{n,t}(x))$ for any $x \neq a$ based on which connected component (after removing $a$) $x$ is a member of. Again, we are interested in components of infinite cardinality.
 
     Let $\left\{\widetilde{G}_k:=(\widetilde{V}_k,\widetilde{E}_k) \; : \; 1 \leq k \leq m\right\}$ be the infinite connected components of the graph $G$ after removing vertex $a$. For any $1 \leq k \leq m$ and distinct $y_1, y_2 \in \widetilde{V}_k \cap \partial V_{(1)}$, let $\Gamma_{y_1,y_2}$ be a path in $\widetilde{G}_k$ connecting $y_1$ and $y_2$. Set 
	$$ D_k := \inf \left\{n \geq 1 : \Gamma_{y_1,y_2} \cap \partial V_{(n)}= \emptyset, \; \; \forall \; y_1 \neq y_2 \in \widetilde{V}_k \cap \partial V_{(1)}\right\} < \infty, \; \; D_{\text{max}}:= \max_{1 \leq k \leq m} D_k.$$
	In other words, the path $\Gamma_{y_1,y_2}$ lies completely in $V_{(D_{\text{max}}-1)}$ for any $y_1,y_2 \in \widetilde{V}_k \cap \partial V_{(1)}$. Finally, let $G^*_k= (V_k^*,E_k^*)$ be the subgraph induced by $V_k^*:=\widetilde{V}_k \cup \left\{a\right\}$. Define, $\mathcal{R}_{n,t,k}$ to be the effective resistance between the vertex $a$ and $V_k^* \cap \partial V_{(n)}= \widetilde{V}_k \cap \partial V_{(n)}$ when the conductance configuration is given by $C_t\big\rvert_{E_k^* \cap E_{(n)}}$. 
	
	\begin{lemma}{\label{voltlower}}
		Fix a vertex $x \in V$ such that $x \neq a$. Suppose that $x \in \widetilde{V}_k$ for some $k \in \left\{1, \ldots, m\right\}$. Then for any $t \geq 0$ and $n \geq D_{\text{max}}$, we have 
		$$(1-v_{n,t}(x)) \geq \min_{y \in \widetilde{V}_k \cap \partial V_{(1)}} (1-v_{n,t}(y)) \geq \dfrac{\Pi_{t,D_{\text{max}}}\cdot\inf_{e \in E_{(1)}} R_t(e)}{|\partial V_{(1)}|\mathcal{R}_{n,t,k}},$$
		where 
		$$ \Pi_{t,\ell}:= \dfrac{\inf_{e \in E_{(\ell)}} C_t(e)}{\sup_{z \in V_{(\ell)}} \operatorname{deg}(z)\cdot  \sup_{e \in E_{(\ell)}} C_t(e)}, \; \forall \; \ell \geq 1.$$
	\end{lemma}
	 
	 \begin{proof}
	 	Since any path from $x$ to $a$ must pass through $\widetilde{V}_k \cap \partial V_{(1)}$, by Lemma~\ref{flowdirec} we have $$1- v_{n,t}(x) \geq \min_{y \in \widetilde{V}_k \cap \partial V_{(1)}} (1-v_{n,t}(y)).$$ Fix $y \neq z \in \widetilde{V}_k \cap \partial V_{(1)}$ and consider any path $\Gamma_{y,z}=(z_0, z_1, \ldots, z_{\ell})$ connecting $y=z_0$ to $z=z_{\ell}$ which lies completely inside $V_{(D_{\text{max}}-1)}$ and avoids $a$.  Since $(1-v_{n,t}(y))$ (resp.~$(1-v_{n,t}(z))$) is the probability that the weighted random walk on $(G_{(n)},C_t)$, starting from $y$ (resp.~$z$), will hit $\partial V_{(n)}$ before visiting $a$, we can write the following for any $n \geq D_{\text{max}}$.
	 	\begin{align*}
	 		(1-v_{n,t}(y)) \geq (1-v_{n,t}(z)) \prod_{\ell^{\prime}=0}^{\ell-1} \dfrac{C_t(z_{\ell^{\prime}},z_{\ell^{\prime}+1})}{\sum_{z^{\prime} : z^{\prime} \sim z_{\ell^{\prime}}} C_t(z_{\ell^{\prime}},z^{\prime})} &\geq (1-v_{n,t}(z)) \dfrac{\inf_{e \in E_{(D_{\text{max}})}} C_t(e)}{\sup_{z^{\prime} \in V_{(D_{\text{max}})}} \operatorname{deg}(z^{\prime}) \sup_{e \in E_{(D_{\text{max}})}} C_t(e)},
	 	\end{align*} 
	 	and thus
	 	$$ \min_{y \in \widetilde{V}_k \cap \partial V_{(1)}} (1-v_{n,t}(y)) \geq \Pi_{t,D_{\text{max}}} \max_{y \in \widetilde{V}_k \cap \partial V_{(1)}} (1-v_{n,t}(y)).$$
	 	Now let $i_{n,t,k}$ be the unit current on $\left(V^*_k \cap V_{(n)}, C_t \big \rvert_{E^*_k \cap E_{(n)}}\right)$ with $a$ as the source and $V^*_k \cap \partial V_{(n)}$ as the sink. Then, $\mathcal{R}_{n,t,k}(1-v_{n,t}(y)) = i_{n,t,k}(a,y) R_t(a,y)$ for any $y \in V^*_k \cap \partial V_{(1)}$. This holds true since $v_{n,t}(y)$ equals the voltage at $y$ in this new electrical network, as argued before Lemma~\ref{supersub}. Since $i_{n,t,k}(a,y) \geq 0$ for $y \in V^*_k \cap \partial V_{(1)}$ and $\sum_{y \in V^*_k \cap \partial V_{(1)}} i_{n,t,k}(a,y)=1$, we have 
	 	\begin{align*}
	 		|\partial V_{(1)}| \max_{y \in \widetilde{V}_k \cap \partial V_{(1)}} (1-v_{n,t}(y)) \geq \sum_{y \in V^*_k \cap \partial V_{(1)}} (1-v_{n,t}(y)) \geq \sum_{y \in V^*_k \cap \partial V_{(1)}} \dfrac{i_{n,t,k}(a,y) R_t(a,y)}{\mathcal{R}_{n,t,k}} \geq \dfrac{\inf_{e \in E_{(1)}} R_t(e)}{\mathcal{R}_{n,t,k}}.
	 	\end{align*}
	 	Hence,
	 	$$(1-v_{n,t}(x)) \geq \min_{y \in \widetilde{V}_k \cap \partial V_{(1)}} (1-v_{n,t}(y)) \geq \dfrac{ \Pi_{t,D_{\text{max}}}\cdot \inf_{e \in E_{(1)}} R_t(e)}{|\partial V_{(1)}|\mathcal{R}_{n,t,k}}.$$
	 \end{proof}
	
	\begin{corollary}{\label{maincor}}
		Fix a vertex $x \in V$ such that $x \neq a$. Suppose that $x \in \widetilde{V}_k$ for some $k \in \left\{1, \ldots, m\right\}$. Then for any $t \geq 0$ and $n \geq D_{\text{max}}$, we have 
	
		\begin{align*}
			\left|\frac{1-v_{n,t+1}(x)}{1-v_{n,t}(x)}-1\right|\leq \dfrac{|\partial V_{(1)}| \sup_{e \in E_{(1)}} C_t(e)}{\Pi_{t,D_{\text{max}}}}\sum_{e\in E_k^*}|R_t(e)-R_{t+1}(e)|.
		\end{align*}
	\end{corollary}
	\begin{proof}
		We apply Lemma~\ref{diffes} for the graph $G^*_k=(V^*_k,E^*_k)$ with $a$ as the source/origin and $V^*_k \cap \partial V_{(n)}$ as the sink. Recall that $v_{n,t}(x)$ still equals the voltage at $x$ in this new network when conductance configuration is given by $C_t \big \rvert_{E^*_k}$.  
		Since the right-hand side of Lemma~\ref{diffes} involves unit currents, taking absolute values gives the following for any $x \in V^*_k \cap V_{(n-1)} , x \neq a$ :
		 $$|v_{n,t}(x)-v_{n,t+1}(x)|\leq \frac{1}{\mathcal{R}_{n,t,k}}\sum_{e\in E^*_k}|R_t(e)-R_{t+1}(e)|.$$
		Moreover, the inequality trivially holds for $x\in V^*_k \setminus V_{(n-1)}$ or $x=a$. Combining this with Lemma~\ref{voltlower} gives our desired result.
	\end{proof}
	
	\subsection{Proof of Theorem~\ref{mainthm}}
	We are now ready to prove Theorem~\ref{mainthm}. We shall continue using the notations introduced before the statement of Lemma~\ref{voltlower}.
	
	\begin{proof}[Proof of Theorem~\ref{mainthm} (Recurrent case)]
	
	Since the RWCE is proper, bounded from above and satisfies the condition in~(\ref{condition}), it is also bounded from below (see Remark~\ref{below}) and hence necessarily elliptic. Moreover, the hypotheses on the RWCE remain true for any sub-process $\left\{\l X_t, G_t \r : t \geq T\right\}$ where $T \geq 0$ (recall that the weighted random walk on $(G,C_t)$ is recurrent for all $t \geq 0$, see \Cref{boundbelow}). Thus, it is enough to show that 
	$$ \P[X_t =a \text{ for some } t \geq 0] =1.$$
 
    If $X_0\in V'$ where $V'$ is a finite connected component after removing $a,$ the claim easily follows from ellipticity of the RWCE. Thus, let $K \in \left\{1, \ldots, m\right\}$ be the smallest positive integer $k$ such that $X_0 \in V_k^*$.  Our aim is to use Lemma~\ref{ost1} with $V^*=V_K^*$; note that $K$ is $\mathscr{F}_0$-measurable. First, we check that for any $t\geq 0$ and $x \in V$, we have $v_{n,t}(x)\to 1$ almost surely as $n\to \infty$. Let $d(a,x)=\ell$ and $(x_0,\dots, x_\ell)$ be a path from $a$ to $x,$ with $x_0=a, x_{\ell}=x$. Let $i^{n,t}$ be the unit current on $(G_{(n)},C_t)$ with $a$ as the source and $\partial V_{(n)}$ as the sink.  Then, for $n>\ell$ we have
		\begin{align*}
			1-v_{n,t}(x) = v_{n,t}(a)-v_{n,t}(x)&= \frac{1}{\mathcal{R}_{n,t}}\sum_{k=0}^{\ell-1}i^{n,t}(x_k,x_{k+1})R_t(x_k,x_{k+1})
			\leq \frac{1}{\mathcal{R}_{n,t}}\sum_{k=0}^{\ell-1}R_t(x_k,x_{k+1}).
		\end{align*}
		Since, the RWCE is proper, it is enough to show that $\mathcal{R}_{n,t} \to \infty$ as $n \to \infty$ almost surely. On the otherhand, using the computation carried out in the proof of Lemma~\ref{resistchange}, we have
		$$ |\mathcal{R}_{n,t} - \mathcal{R}_{n,0} | \leq \sum_{e \in E} |R_t(e)-R_0(e)| \leq \sum_{s=0}^{t-1} \sum_{e \in E} |R_{s}(e)-R_{s+1}(e)| \leq \Gamma < \infty,$$
		almost surely by the hypothesis of \Cref{mainthm}, whereas $\mathcal{R}_{n,0} \to \infty$ almost surely (since the weighted random walk on $(G,C_0)$ is recurrent). This shows that $v_{n,t}(x) \to 1$ almost surely as $n \to \infty$.  
		
		Next, we show that the condition on $\alpha^*_{n,t}$ holds true with $V^*=V_K^*$. 
	For any $\ell \geq 1$ and $t \geq 0$, 
		$$ \Pi_{t,\ell} \geq \dfrac{\inf_{e \in E_{(\ell)}} \inf_{s \geq 0} C_s(e)}{\sup_{z \in V_{(\ell)}} \operatorname{deg}(z) \sup_{e \in E_{(\ell)}} \sup_{s \geq 0} C_s(e)}.$$
		Therefore, by Corollary~\ref{maincor}, we get 
		\begin{align*}
		\sup_{n \geq D_{\text{max}}}	\prod_{t=0}^\infty \alpha^*_{n,t} &\leq \prod_{t=0}^\infty \left (1 +  \dfrac{|\partial V_{(1)}| \sup_{e \in E_{(1)}} C_t(e)}{\Pi_{t,D_{\text{max}}}}\sum_{e\in E_K^*}|R_t(e)-R_{t+1}(e)|\right) \\
		&	\leq \exp\left(\sum_{t} \dfrac{|\partial V_{(1)}| \sup_{e \in E_{(1)}} C_t(e)}{\Pi_{t,D_{\text{max}}}} \sum_{e \in E} |R_t(e)-R_{t+1}(e)|\right)\\
		& \leq \exp\left( \Gamma \cdot\dfrac{|\partial V_{(1)}| \cdot\sup_{z \in V_{(D_{\text{max}})}} \operatorname{deg}(z) \cdot \sup_{e \in E_{(1)}, t\geq 0} C_t(e) \cdot \sup_{e \in E_{(D_{\text{max}})}, t\geq 0} C_t(e)}{\inf_{e \in E_{(D_{\text{max}})}, t \geq 0} C_t(e)}\right)< \infty,
		\end{align*}
	almost surely. We now apply Lemma~\ref{ost1} to conclude that $\P[X_t=a \text{ for some } t \geq 0 \mid \mathscr{F}_0] =1$ almost surely on the event $(X_0 \in V^*_K)$. This completes the proof. 
		\end{proof}
  
	\begin{proof}[Proof of Theorem~\ref{mainthm}(Transient case)]
		We aim to use Lemma~\ref{ost2}. For large enough $n,$ note that no current flows in the finite components after removing $a$ when considering $\partial V_{(n)}$ as the sink. In other words, for large enough $n$, the effective resistances in these finite components are infinity and 
		thus, by parallel law of resistances, we have $ 1/\mathcal{R}_{n,t} = \sum_{k=1}^m 1/\mathcal{R}_{n,t,k}$. Taking $n \to \infty$, we have $ 1/\mathcal{R}_{\infty,t} = \sum_{k=1}^m 1/\mathcal{R}_{\infty,t,k},$ where $\mathcal{R}_{\infty,t,k}$ is the almost sure limit of $\mathcal{R}_{n,t,k}$ as $n \to \infty$. Since the weighted random walk on $(G,C_0)$ is transient, we have $\mathcal{R}_{\infty,0}< \infty$ almost surely and hence almost surely at least one $\mathcal{R}_{\infty,0,k}$ is finite. Let $K \in \left\{1,\ldots,m\right\}$ be the smallest positive integer $k$ such that $\mathcal{R}_{\infty,0,k}$ is finite. Note that both $K$ and $\mathcal{R}_{\infty,0,K}$ are $\mathscr{F}_0$-measurable.
		
			Recall from the proof of the recurrent case that, for any $t \geq 0$, 
		$$ \dfrac{|\partial V_{(1)}| \sup_{e \in E_{(1)}} C_t(e)}{\Pi_{t,D_{\text{max}}}} \leq \dfrac{|\partial V_{(1)}|\cdot  \sup_{z \in V_{(D_{\text{max}})}} \operatorname{deg}(z) \cdot \sup_{e \in E_{(1)}, s\geq 0} C_s(e) \cdot \sup_{e \in E_{(D_{\text{max}})}, s\geq 0} C_s(e)}{\inf_{e \in E_{(D_{\text{max}})}, s \geq 0} C_s(e)} =: \Gamma^* < \infty,$$
		almost surely. We want to apply Lemma~\ref{ost2} with $V^*= V_K^*$. Towards that end we define for any $t \geq 0$, 
		$$ \Gamma_t^* :=	\dfrac{|\partial V_{(1)}| \cdot\sup_{z \in V_{(D_{\text{max}})}} \operatorname{deg}(z) \cdot \sup_{e \in E_{(1)}} C_t(e) \cdot\sup_{e \in E_{(D_{\text{max}})}} C_t(e)}{\inf_{e \in E_{(D_{\text{max}})}} C_t(e)},$$
		and for any $t \geq 1$, $N \in \mathbb{N}$,
		 $$ \Gamma_t := \sum_{s=0}^{t-1} \sum_{e \in E} \big \rvert R_{s}(e)-R_{s+1}(e)\big \rvert, \;\; \Lambda_{N,t}:= \inf_{n \geq N} \prod_{s=0}^{t-1}\beta_{n,s}^*.$$
		 Note that $\Gamma_t^*, \Gamma_t, \Lambda_{N,t}$ are all $\mathscr{F}_t$-measurable. Moreover, $\sup_{t \geq 0} \Gamma^*_t \leq \Gamma^* < \infty $ and $\sup_{t \geq 1} \Gamma_t = \Gamma < \infty$ almost surely. Also, 
		 $$ \inf_{t \geq 1}  \Lambda_{N,t} =   \inf_{t \geq 1} \inf_{n \geq N} \prod_{s=0}^{t-1} \beta^*_{n,s} =   \inf_{n \geq N} \inf_{t \geq 1} \prod_{s=0}^{t-1} \beta^*_{n,s} = \inf_{n \geq N} \prod_{s \geq 0} \beta^*_{n,s}.$$   
		
		We start the proof by first considering the case where there exists $\delta >0$ (non-random) such that 
		$$ \inf_{n \geq D_{\text{max}}} \prod_{t \geq 0} \beta^*_{n,t} =\inf_{t \geq 1} \Lambda_{D_{\text{max}},t} \geq \delta,\;\; \text{ and } \max \left(\Gamma, \sup_{t \geq 1}\Gamma^*_t\right) \leq 1/\delta$$
		almost surely. Similar to the recurrent case, it is enough to show that $\P[X_t =a, \text{ finitely often}]=1$. Given $M \in (0,\infty)$ and $n \geq D_{\text{max}}$, note that on the event $A_M:=\left(\mathcal{R}_{\infty,0,K}, \Gamma_0^* \leq M \right) \in \mathscr{F}_0$ Lemma~\ref{voltlower} yields the following for any $t \geq 0$: 
		$$ 	 \inf_{ x \in V_K^* \setminus \left\{a\right\}} (1- v_{n,t}(x)) \geq   \dfrac{1}{\Gamma_t^* \mathcal{R}_{n,t,K}} \geq  \dfrac{1}{\Gamma_t^* \mathcal{R}_{\infty,t,K}}.$$
		Using the proof of Lemma~\ref{resistchange}, we have 
		$$ |\mathcal{R}_{n,t,K} - \mathcal{R}_{n,0,K} | \leq \sum_{e \in E_K^*} |R_t(e)-R_0(e)| \leq \sum_{s=0}^{t-1} \sum_{e \in E} |R_{s}(e)-R_{s+1}(e)| \leq \Gamma \leq 1/\delta,$$
		and thus $|\mathcal{R}_{\infty,t,K} - \mathcal{R}_{\infty,0,K} | \leq 1/\delta$. Therefore, on the event $A_M$, we have for any $t \geq 0$, 
		$$  \inf_{ x \in V_K^* \setminus \left\{a\right\}, n \geq D_{\text{max}}} (1- v_{n,t}(x)) \geq   \dfrac{1}{\Gamma_t^* \mathcal{R}_{\infty,t,K}} \geq \dfrac{1}{(M \vee \delta^{-1}) (M+\delta^{-1})} \geq (M+\delta^{-1})^{-2}.$$
		Applying Lemma~\ref{ost2}, we can arrive at the following conclusion : For any $T \geq 0$, 
		$$ \P[X_t \neq a, \; \forall \; t \geq T \mid \mathcal{F}_T] \geq \delta(M+1/\delta)^{-2}>0,$$
		almost surely on $A_M \cap(X_T \in V_K^*\setminus \{a\})$. \textit{Levy's upward theorem} now implies that almost surely, 
		\begin{align*}
			\mathbbm{1}_{(X_t =a \text{ finitely often})} &\geq \delta(M+1/\delta)^{-2} \limsup_{T \to \infty} \mathbbm{1}_{A_M \cap (X_T \in V_K^*\setminus \{a\})} \\
			& = \delta(M+1/\delta)^{-2} \mathbbm{1}_{A_M \cap (X_t \in V_K^*\setminus \{a\}\text{ infinitely often})} \\
			& \geq \delta(M+1/\delta)^{-2} \mathbbm{1}_{A_M \cap (X_t = a \text{ infinitely often})},
		\end{align*} 
		where the last inequality follows from ellipticity of the RWCE. Thus, $\P[X_t=a \text{ infinitely often}, A_M] =0$ for any $M \in (0,\infty)$. Taking $M \to \infty$ and recalling that $\mathcal{R}_{\infty, 0, K}, \Gamma_0^*< \infty$ almost surely, we arrive at 
	 $\P[X_t =a, \text{ finitely often}]=1$. 
		
		For the general case, we first set $\upsilon:= \inf \left\{ t \geq 0 :  \sum_{s \geq t} \sum_{e \in E} |R_{s}(e)-R_{s+1}(e)| \leq 1/(2\Gamma^*) \right\}$. By the hypothesis in~(\ref{condition}) and the fact that $\Gamma^* < \infty$ almost surely, we know that $\upsilon < \infty$ almost surely. 
		
		Note that, for any $t \geq \upsilon$, we have $\Gamma^* \sum_{e \in E} |R_t(e)-R_{t+1}(e)|  \leq 1/2$, almost surely. Since, $\log(1-x)+2x \geq 0$ for $x \in [0,1/2]$, Corollary~\ref{maincor} yields the following: 
		\begin{align*}
			\inf_{n \geq D_{\text{max}}} \prod_{t \geq  \upsilon} \beta^*_{n,t} \geq   \prod_{t \geq  \upsilon} \left( 1- \Gamma_t^* \sum_{e \in E} \big \rvert R_t(e)-R_{t+1}(e) \big \rvert\right)_+  &\geq  \prod_{t \geq  \upsilon} \exp\left( -2\Gamma^* \sum_{e \in E} \big \rvert R_t(e)-R_{t+1}(e) \big \rvert\right) \\
			&\geq \exp(-2\Gamma^*\Gamma) >0,
		\end{align*} 
		almost surely, where $t_+:=t \lor 0$. 
		On the other hand, 
		\begin{align}
			\beta^*_{n,t}\geq \inf_{m \geq D_{\text{max}}, s \geq 0, x \in V_K^* \setminus \left\{a\right\}} (1- v_{m,s}(x)) \geq   \dfrac{1}{\Gamma^* \mathcal{R}_{\infty,t,K}} \geq \dfrac{1}{\Gamma^*(\mathcal{R}_{\infty,0,K}+\Gamma)},
		\end{align} 
		and thus 
		$$ \inf_{n \geq D_{\text{max}}} \prod_{t < \upsilon} \beta^*_{n,t} \geq \left( \dfrac{1}{\Gamma^*(\mathcal{R}_{\infty,0,K}+\Gamma)}\right)^{\upsilon} >0,$$
		almost surely. Combining the above estimates, we conclude that $\inf_{t \geq 1} \Lambda_{D_{\text{max}},t} = \inf_{n \geq D_{\text{max}}} \prod_{t \geq 0} \beta_{n,t}^* >0$ almost surely. Now we define a sequence of $\mathscr{F}_t$-stopping times :
	$$ \gamma_m := \inf \left\{ t \geq 1 :  \Lambda_{D_{\text{max}},t} \leq 1/m \text{ or } \max(\Gamma_t,\Gamma^*_t) \geq m\right\}.$$
	for any $m \geq 1$. We construct an RWCE where we freeze the evolution of the environment at time $\gamma_m-1$. The rigorous construction goes as follows. Given $\mathscr{F}_{\infty} = \sigma(\l X_s, G_s \r : s \geq 0)$, for each $t \geq 1$, set $\left\{Y_s^{(t)} : s \geq 0\right\}$ to be a weighted random walk on $(G, C_{t-1})$ with $Y_0^{(t)}=X_t$. For all $m \geq 1$, we then define,
		\begin{equation}{\label{defstop}}
		\widetilde{X}^{(m)}_t := \begin{cases} X_t, & \text{ if } t \leq \gamma_m, \\ Y_{t-\gamma_m}^{(\gamma_m)}, & \text{ if } t > \gamma_m,\end{cases}
	\end{equation}
	whereas $\widetilde{G}^{(m)}_t = (V,E, \widetilde{C}^{(m)}_t)$ with $\widetilde{C}^{(m)}_t = C_{(\gamma_m-1) \wedge t}$. We claim that $\left\{\l \widetilde{X}_t^{(m)}, \widetilde{G}_t^{(m)} \r\right\}_{t \geq 0}$ is indeed an RWCE. The proof of this claim is deferred to Section~\ref{appn}. Since the original walk is proper and bounded (from above and below), the new walks also satisfy those properties trivially and hence are necessarily elliptic. If $\gamma_m=1$, then the walk $\left\{\widetilde{X}_t^{(m)} : t \geq 0\right\}$, conditioned on $\mathcal{F}_{\infty}$, is just a weighted random walk on $(G,C_0)$ and hence transient. If  $\gamma_m >1$, then 
	\begin{align*}
	\sup_{t \geq 1}	 \widetilde{\Gamma}^{*(m)}_t&:= \sup_{t \geq 1} \dfrac{|\partial V_{(1)}| \cdot \sup_{z \in V_{(D_{\text{max}})}} \operatorname{deg}(z) \cdot\sup_{e \in E_{(1)}} \widetilde{C}^{(m)}_t(e) \cdot \sup_{e \in E_{(D_{\text{max}})}} \widetilde{C}^{(m)}_t(e)}{\inf_{e \in E_{(D_{\text{max}})}} \widetilde{C}^{(m)}_t(e)} = \sup_{1 \leq t \leq \gamma_m-1} \Gamma^*_t \leq m,
	\end{align*}
	whereas
	$$ \sum_{t \geq 0} \sum_{e} \big \rvert \widetilde{R}_t^{(m)}(e)-\widetilde{R}_{t+1}^{(m)}(e) \big \rvert = \sum_{t=0}^{\gamma_m-2} \sum_{e} \big \rvert R_t(e) - R_{t+1}(e) \big \rvert =\Gamma_{\gamma_m-1}\leq m.$$ 
	Moreover, if $\widetilde{\Lambda}^{(m)}_{N,t}$ is defined for $\left\{\widetilde{C}^{(m)}_s : s \geq 0\right\}$ in a fashion identical to how $\Lambda_{N,t}$ was defined for $\left\{C_s : s \geq 0\right\}$, we have 
	$$ \inf_{t \geq 1} \widetilde{\Lambda}^{(m)}_{D_{\text{max}},t} = \inf_{1 \leq t \leq \gamma_m-1} \widetilde{\Lambda}^{(m)}_{D_{\text{max}},t} = \inf_{1 \leq t \leq \gamma_m-1} \Lambda_{D_{\text{max}},t} \geq 1/m. $$
	Finally, since $\widetilde{C}_0^{(m)} = C_0$ and the weighted random walk on $(G,C_0)$ is assumed to be transient, we conclude by the first part of this proof that the RWCE $\left\{\l \widetilde{X}_t^{(m)}, \widetilde{G}_t^{(m)} \r\right\}_{t \geq 0}$ is transient. In other words, for any $m\geq 1$ we have $\P \left[\widetilde{X}_t^{(m)} =a \text{ for finitely many }t \geq 0\right]=1$ and thus
	$$\P\left[\widetilde{X}_t^{(m)} =a \text{ for finitely many }t \geq 0, \text{ for all } m \geq 1\right]=1.$$ To conclude, note that $\sup_{t \geq 0}\Gamma_t=\Gamma < \infty, \;\sup_{t \geq 0} \Gamma_t^* \leq \Gamma^* < \infty$ and $\inf_{t \geq 1} \Lambda_{D_{\text{max}},t} >0$ almost surely, which guarantees that $\gamma_{m^*} =\infty$ for some $m^*$ and hence $\widetilde{X}_t^{(m^*)}=X_t$ for all $t \geq 0$. This shows that $\P[X_t=a \text{ for finitely many }t \geq 0]=1$, which completes the proof.
	
		\end{proof}

	\section{Acknowledgment}
	A major portion of this work is taken from the Honors thesis of the first author submitted to the Department of mathematics at Stanford University. The authors would like to thank Amir Dembo for introducing this interesting topic to them and his useful comments throughout the progress of this project.

	\section{Appendix}{\label{appn}}
	
		\begin{proof}[Proof of Lemma~\ref{nonadaptive}] We prove the lemma by induction on $t$. The statement is trivially true for $t=0$ since $X_0=X_0^{\prime}$. Suppose that it is true for $t \leq T-1$ for some $T \geq 1$. Fix $x_0, \ldots, x_T \in V$ such that $x_{i-1} \sim x_i$ for all $i=1,\ldots,T$. Also fix measurable $\A_0, \ldots, \A_T \subseteq \mathcal{W}$. It is easy to see that 
		\begin{align*}
			\mathbb{P} \left[ X^{\prime}_s=x_s, C_s \in \A_s, \; \forall \; 0 \leq s \leq T \right] &= \mathbb{E} \left[ 	\mathbb{P} \left( X^{\prime}_s=x_s, C_s \in \A_s, \; \forall \; 0 \leq s \leq T \mid \sigma\left(\mathscr{G}_{\infty},X_{0}^{\prime}\right)\right) \right] \\
			&= \mathbb{E} \left[  \prod_{s=0}^{T-1} P(x_s,x_{s+1};C_s) \mathbbm{1}_{(X_0^{\prime}=x_0)}\prod_{s=0}^T \mathbbm{1}_{(C_s \in \A_s)} \right], \\
			&= \mathbb{E} \left[  \prod_{s=0}^{T-1} P(x_s,x_{s+1};C_s) \mathbbm{1}_{(X_0=x_0)}\prod_{s=0}^T \mathbbm{1}_{(C_s \in \A_s)} \right].
		\end{align*} 
		where the last equality follows from the fact that $X_0=X_0^{\prime}$. 
		On the otherhand, 
		\begin{align*}
			\mathbb{P} \left[ X_s=x_s, C_s \in \A_s, \; \forall \; 0 \leq s \leq T \right] 
			& = \mathbb{E} \left[ 	\mathbb{P} \left( X_s=x_s, C_s \in \A_s, \; \forall \; 0 \leq s \leq T\mid \sigma\left(\mathscr{F}_{T-1},X_{T}\right)\right) \right] \\
			& = \mathbb{E} \left[ \mathbb{P}\left(C_T \in \A_T \mid \sigma\left(\mathscr{F}_{T-1},X_{T}\right)\right) \prod_{s=0}^T \mathbbm{1}_{(X_{s}=x_s)} \prod_{s=0}^{T-1} \mathbbm{1}_{(C_s \in \A_s)}\right] \\
			& = \mathbb{E} \left[ \mathbb{P}\left(C_T \in \A_T \mid \mathscr{G}_{T-1}\right) \prod_{s=0}^T \mathbbm{1}_{(X_{s}=x_s)} \prod_{s=0}^{T-1} \mathbbm{1}_{(C_s \in \A_s)}\right] \\
			& = \mathbb{E} \left[ \mathbb{E} \left(\mathbb{P}\left(C_T \in \A_T \mid \mathscr{G}_{T-1}\right) \prod_{s=0}^{T} \mathbbm{1}_{(X_{s}=x_s)} \prod_{s=0}^{T-1} \mathbbm{1}_{(C_s \in \A_s)} \Bigg \rvert \mathscr{F}_{T-1} \right)\right] \\
			&= \mathbb{E} \left[ P(x_{T-1},x_T;C_{T-1})\mathbb{P}\left(C_T \in \A_T \mid \mathscr{G}_{T-1}\right) \prod_{s=0}^{T-1} \mathbbm{1}_{(X_{s}=x_s, C_s \in \A_s)} \right],
		\end{align*} 
		where the third equality follows from the assumption of non-adaptiveness. Now we apply the induction hypothesis for $t=T-1$. Observing that $\mathbb{P}\left(C_T \in \A_T \mid \mathscr{G}_{T-1}\right)$ is a measurable function of $(C_0,\ldots,C_{T-1})$, we can write 
		\begin{align*}
			\mathbb{P} \left[ X_s=x_s, C_s \in \A_s, \; \forall \; s \leq T \right] 
			& = \mathbb{E} \left[ P(x_{T-1},x_T;C_{T-1}) \mathbb{P}\left(C_T \in \A_T \mid \mathscr{G}_{T-1}\right) \prod_{s=0}^{T-1} \mathbbm{1}_{(X^{\prime}_{s}=x_s, C_s \in \A_s)} \right] \\
			&=  \mathbb{E} \left[ \prod_{s=0}^{T-1} P(x_s,x_{s+1};C_s) \mathbb{P}\left(C_T \in \A_T \mid \mathscr{G}_{T-1}\right) \mathbbm{1}_{(X_0^{\prime}=x_0)}\prod_{s=0}^{T-1} \mathbbm{1}_{(C_s \in \A_s)} \right] \\
			&=  \mathbb{E} \left[ \prod_{s=0}^{T-1} P(x_s,x_{s+1};C_s) \mathbb{E}\left(\mathbbm{1}_{(C_T \in \A_T)} \mid \mathscr{F}_{T-1}\right) \mathbbm{1}_{(X_0=x_0)}\prod_{s=0}^{T-1} \mathbbm{1}_{(C_s \in \A_s)} \right] \\
			&= \mathbb{E} \left[ \prod_{s=0}^{T-1} P(x_s,x_{s+1};C_s) \mathbbm{1}_{(X_0=x_0)}\prod_{s=0}^{T} \mathbbm{1}_{(C_s \in \A_s)} \right],
		\end{align*}
		where the second equality follows from the first one after taking conditional expectation with respect to $\sigma(\mathscr{G}_{\infty},X_0^{\prime})$ and the third equality follows again form the non-adaptiveness condition. The final equality concludes the proof by induction. 
	\end{proof}

	\begin{lemma}{\label{flowdirec}}
		Let $G=(V,E)$ be a connected finite graph equipped with a non-zero finite conductance configuration $C$. Fix $u\in V$ and non-empty $S, A \subseteq V$ such that $\left\{u \right\}, S, A $ are all mutually disjoint. Let $v(x)$ be the voltage at vertex $x$ if we ground $A$ and put $u$ at voltage $1$. For any $y \in V$, if every path from vertex $y$ to vertex $u$ in $G$ has to go through $S$, then $ \max_{x \in S} v(x) \geq v(y)$.
	\end{lemma}
	
\begin{proof}
		We use the fact that for any $x \in V$, $v(x)$ is the probability that the weighted random walk on $(G,C)$, started from $x$, shall reach $u$ before hitting $A$. Let $\left\{Y_m : t \geq 0 \right\}$ be the weighted random walk with $Y_0=y,$ and $\mathscr{H}_t:= \sigma(Y_s : s \leq t)$.  Let $\theta$ be the first time the walk reaches $u$ or $A$, and $\theta_S$ be the first time it hits $S$.  Note that, $\theta$ and $\theta_S$ are both finite almost surely since the graph is finite and the conductances are all positive real numbers. Since all paths from $y$ to $u$ passes through $S$, we have the following.
		\begin{align*}
			v(y) = \mathbb{P} [Y_{\theta}=u] = \mathbb{P} [Y_{\theta}=u, \theta_S \leq \theta] 
			&= \mathbb{E} \left[ \mathbbm{1}_{(\theta_S \leq \theta)} \mathbb{P}(Y_{\theta}=u|\mathscr{H}_{\theta \wedge \theta_S}) \right] \\
			&= \mathbb{E} \left[ \mathbbm{1}_{(\theta_S \leq \theta)} \mathbb{P}(Y_{\theta}=u|\mathscr{H}_{\theta_S}) \right] \\
			&= \mathbb{E} \left[ \mathbbm{1}_{(\theta_S \leq \theta)} \sum_{x \in S} \mathbb{P}(Y_{\theta}=u|\mathscr{H}_{\theta_S})\mathbbm{1}_{(Y_{\theta_S}=x)}  \right] \\
			&= \mathbb{E} \left[ \mathbbm{1}_{(\theta_S \leq \theta)} \sum_{x \in S} v(x)\mathbbm{1}_{(Y_{\theta_S}=x)} \right] \\
			&\leq \left( \max_{x \in S} v(x) \right) \mathbb{P}[\theta_S \leq \theta] \leq \max_{x \in S} v(x).
		\end{align*}
		This completes the proof.
	\end{proof}
	
	\begin{proof}[Addendum to the proof of Theorem~\ref{mainthm}: transient case]
		Fix $t \geq 0$, $x_0, \ldots, x_t \in V$ and measurable $\A_0, \ldots, \A_t \subseteq \W$. Then, for any $y \sim x_t$, we have
		\begin{align*}
			\E \left[  \mathbbm{1}_{(\widetilde{X}^{(m)}_{t+1}=y)}\prod_{s=0}^{t} \mathbbm{1}_{(\widetilde{X}^{(m)}_s=x_s,\widetilde{C}^{(m)}_s \in \A_s)} \mathbbm{1}_{(\gamma_m > t)} \right]	& = 	 \E \left[  \mathbbm{1}_{(X_{t+1}=y)}\prod_{s=0}^{t} \mathbbm{1}_{(X_s=x_s,C_s \in \A_s)} \mathbbm{1}_{(\gamma_m > t)} \right] \\
			&=\E \left[  P(x_t,y; C_t) \prod_{s=0}^t \mathbbm{1}_{(X_s=x_s,C_s \in \A_s)} \mathbbm{1}_{(\gamma_m > t)} \right] \\
			&= \E \left[  P\left(\widetilde{X}_t^{(m)},y; \widetilde{C}_t^{(m)}\right)\prod_{s=0}^{t} \mathbbm{1}_{(\widetilde{X}^{(m)}_s=x_s,\widetilde{C}^{(m)}_s \in \A_s)} \mathbbm{1}_{(\gamma_m > t)} \right].
		\end{align*}
		Similarly, we see that
		\begin{align*}
			\E \left[  \mathbbm{1}_{(\widetilde{X}^{(m)}_{t+1}=y)}\prod_{s=0}^{t} \mathbbm{1}_{(\widetilde{X}^{(m)}_s=x_s,\widetilde{C}^{(m)}_s \in \A_s)} \mathbbm{1}_{(\gamma_m \leq t)} \right]	& = 	\sum_{u=1}^t \E \left[  \mathbbm{1}_{(\widetilde{X}^{(m)}_{t+1}=y)}\prod_{s=0}^{t} \mathbbm{1}_{(\widetilde{X}^{(m)}_s=x_s,\widetilde{C}^{(m)}_s \in \A_s)} \mathbbm{1}_{(\gamma_m =u )} \right]	\\
			&=	\E \sum_{u=1}^t \mathbbm{1}_{(Y_{t+1-u}^{(u)}=y)} \left(\prod_{s=0}^{u-1}  \mathbbm{1}_{(X_s=x_s,C_s \in \A_s)} \right)\\
			& \hspace{1.5 in} \left(\prod_{s=u}^{t} \mathbbm{1}_{(Y_{s-u}^{(u)}=x_s,C_{u-1} \in \A_s)} \right) \mathbbm{1}_{(\gamma_m=u)} \\
			&=	\E \sum_{u=1}^t P(x_t,y; C_{u-1}) \left(\prod_{s=0}^{u-1}  \mathbbm{1}_{(X_s=x_s,C_s \in \A_s)} \right)\\
			& \hspace{1.5 in} \left(\prod_{s=u}^{t} \mathbbm{1}_{(Y_{s-u}^{(u)}=x_s,C_{u-1} \in \A_s)} \right) \mathbbm{1}_{(\gamma_m=u)} \\
			&= \E \left[  P\left(\widetilde{X}_t^{(m)},y; \widetilde{C}_t^{(m)}\right)\prod_{s=0}^{t} \mathbbm{1}_{(\widetilde{X}^{(m)}_s=x_s,\widetilde{C}^{(m)}_s \in \A_s)} \mathbbm{1}_{(\gamma_m \leq t)} \right]
		\end{align*}
		where the third equality is obtained by taking the conditional expectation on $\sigma(\mathscr{F}_{\infty}; Y_s^{(u)} : s \leq t-u)$. This shows that 
		$$ \P\left[\widetilde{X}_{t+1}^{(m)} =y \Big \rvert \sigma\left( \l \widetilde{X}_s^{(m)}, \widetilde{G}_s^{(m)} \r : s \leq t \right) \right]= P( \widetilde{X}_t^{(m)},y; \widetilde{C}_t^{(m)})$$
	and thus $\left\{\l \widetilde{X}_t^{(m)}, \widetilde{G}_t^{(m)} \r\right\}_{t \geq 0}$ is an RWCE.
		
	\end{proof}

\end{document}